\documentclass[a4paper]{amsart}
\usepackage{amsfonts,amsmath,amssymb,amsthm}

\newcommand{\meantmp}[2]{#1\langle{#2}#1\rangle}
\newcommand{\mean}[1]{\meantmp{}{#1}}

\numberwithin{equation}{section}

\usepackage{graphicx}
\usepackage{color}

\newcommand{\R}{\mathbb{R}}
\newcommand{\N}{\mathbb{N}}
\newcommand{\Z}{\mathbb{Z}}

\newcommand{\bu}{\mathbf{u}}
\newcommand{\bftau}{\mathbf{\tau}}
\newcommand{\bq}{\mathbf{q}}
\newcommand{\bfp}{\mathbf{p}}
\newcommand{\bfg}{\mathbf{g}}

\newcommand{\dt}{{\Delta t}}

\newcommand{\by}{{y}}
\newcommand{\ba}{\mathbf{a}}

\newcommand{\bn}{\mathbf{n}}

\newcommand{\bG}{\boldsymbol{G}}
\newcommand{\bR}{\boldsymbol{R}}
\newcommand{\bpsi}{\boldsymbol{\psi}}

\newcommand{\bvphi}{{\boldsymbol{\varphi}}}

\newcommand{\bsigma}{\boldsymbol{\sigma}}
\DeclareMathOperator{\divg}{div}
\DeclareMathOperator{\bog}{Bog}

\DeclareMathOperator{\test}{Test_{\eta}}
\DeclareMathOperator{\testn}{Test_{\eta^{\varepsilon_n}}}
\DeclareMathOperator{\sym}{sym}
\DeclareMathOperator{\cor}{\Kcal_{\eta}}
  \newcommand{\seb}[1]{\textcolor[rgb]{0.00,0.00,1.00}{#1}}

\newtheoremstyle{thmlemcorr}{10pt}{10pt}{\itshape}{}{\bfseries}{.}{10pt}{{\thmname{#1}\thmnumber{ #2}\thmnote{ (#3)}}}
\newtheoremstyle{thmlemcorr*}{10pt}{10pt}{\itshape}{}{\bfseries}{.}\newline{{\thmname{#1}\thmnumber{ #2}\thmnote{ (#3)}}}
\newtheoremstyle{remexample}{10pt}{10pt}{}{}{\bfseries}{.}{10pt}{{\thmname{#1}\thmnumber{ #2}\thmnote{ (#3)}}}

\theoremstyle{thmlemcorr}
\newtheorem{theorem}{Theorem}
\numberwithin{theorem}{section}
\newtheorem{lemma}[theorem]{Lemma}
\newtheorem{corollary}[theorem]{Corollary}
\newtheorem{proposition}[theorem]{Proposition}

\newtheorem{definition}[theorem]{Definition}

\theoremstyle{thmlemcorr*}
\newtheorem{theorem*}{Theorem}
\newtheorem{lemma*}[theorem]{Lemma}
\newtheorem{corollary*}[theorem]{Corollary}
\newtheorem{proposition*}[theorem]{Proposition}
\newtheorem{problem*}[theorem]{Problem}
\newtheorem{conjecture*}[theorem]{Conjecture}
\newtheorem{definition*}[theorem]{Definition}

\theoremstyle{remexample}
\newtheorem{remark}[theorem]{Remark}

\newcommand{\Acal}{\mathcal{A}}

\newcommand{\Ecal}{\mathcal{E}}

\newcommand{\Kcal}{\mathcal{K}}

\newcommand{\Tcal}{\mathcal{T}}

\newcommand{\Vcal}{\mathcal{V}}

\newcommand{\bfPhi}{\mathbf{\Phi}}

\newcommand{\Nbb}{\mathbb{N}}

\newcommand{\Rbb}{\mathbb{R}}

\DeclareMathOperator{\diverg}{div}

\DeclareMathOperator{\supp}{supp}

\newcommand{\norm}[1]{\|#1\|}

\newcommand{\abs}[1]{|#1|}

\newcommand{\absB}[1]{\Bigl|#1\Bigr|}
\newcommand{\absBB}[1]{\biggl|#1\biggr|}

\newcommand{\skp}[1]{\langle #1 \rangle}

\newcommand{\toweak}{\rightharpoonup}
\newcommand{\toweakstar}{\overset{*}\rightharpoonup}

\newcommand{\eps}{\epsilon}

\newcommand{\Bogovskij}{{Bogovski\u{\i}}{}}

\def\Xint#1{\mathchoice 
{\XXint\displaystyle\textstyle{#1}}%
{\XXint\textstyle\scriptstyle{#1}}%
{\XXint\scriptstyle\scriptscriptstyle{#1}}%
{\XXint\scriptscriptstyle\scriptscriptstyle{#1}}%
\!\int} 
\def\XXint#1#2#3{{\setbox0=\hbox{$#1{#2#3}{\int}$} 
\vcenter{\hbox{$#2#3$}}\kern-.5\wd0}} 
\def\dashint{\,\Xint-}
\makeatletter

\begin{document}

\title[Existence and regularity of FSI with non-linear shell]{Existence and regularity of weak solutions for a fluid interacting with a non-linear shell in three dimensions}


\author{{ Boris Muha and Sebastian Schwarzacher }}
\address{Boris Muha, Department of Mathematics, University of Zagreb, Croatia, borism@math.hr; Sebastian Schwarzacher, Department of Mathematics and Physics, Charles University, Prague, Czech Republic, schwarz@karlin.mff.cuni.cz}

\keywords{Fluid-structure interaction, weak solutions, regularity, compactness}

\date{}

\begin{abstract} We study the unsteady incompressible Navier-Stokes equations in three dimensions interacting with a non-linear flexible shell of Koiter type. 
This leads to a coupled system of non-linear PDEs where the moving part of the boundary is an unknown of the problem. 
The known existence theory for weak solutions is extended to non-linear Koiter shell models. We introduce a-priori estimates that reveal higher regularity of the shell displacement beyond energy estimates. These are essential for non-linear Koiter shell models, since such shell models are non-convex (w.r.t.\ terms of highest order). The estimates are obtained by introducing new analytical tools that allow to exploit dissipative effects of the fluid for the (non-dissipative) solid. The regularity result depends on the geometric constitution alone and is independent of the approximation procedure; hence it holds for arbitrary weak solutions. The developed tools are further used to introduce a generalized Aubin-Lions type compactness result suitable for fluid-structure interactions. 
\end{abstract}

\maketitle

\section{Introduction}
Fluid-structure interactions (FSI) are everyday phenomena with many applications, for example in aeroelasticity \cite{Dowel15} and biomedicine \cite{FSIforBIO}. Mathematically, the FSI problems are described by coupling fluid equations with elasticity equations. The analysis of the FSI problems is challenging and attractive mainly due to the following properties. First, the resulting system of non-linear PDEs is of hyperbolic-parabolic type with the coupling taking place at the fluid-structure interface. Second, the fluid domain is an unknown of the problem, i.e.\ the resulting problem is a moving boundary problem. In this paper we study the coupling of the $3D$ incompressible Navier-Stokes equations with the evolution of the non-linear Koiter shell equation. Our main result is that any finite-energy weak solution to the considered FSI problem satisfies an additional regularity property on its interval of existence (see Theorem~\ref{RegTem}). 
More precisely, we show that the elastic displacement belongs to the following Bochner space $L^{2}_t( H^{2+s}_x)\cap H^1_t(H_x^{s})$ for all $s<\frac12$. Here $H^s$ denotes the standard fractional Sobolev space\footnote{For a precise definition of the fractional Sobolev space see Subsection~\ref{ssec:notation}.}. In particular, due to respective embedding theorems, the elastic displacement is Lipschitz continuous in the space variable for almost every moment of time. We use this result to show the existence of weak solutions to a fluid-non-linear Koiter shell interaction problem (see Theorem~\ref{thm:ex}). Since the non-linear Kotier shell equations are quasi-linear with non-linear coefficients depending on the terms of leading order in the energy, the additional structure regularity estimate is crucial for the compactness argument in the construction of a weak solution. The main idea behind the regularity theorem is to use the fluid dissipation and the coupling conditions to prove the additional regularity estimate for the structure displacement. The realization of this idea is technically challenging. It includes the development of a comprehensive analysis to construct a solenoidal extension and smooth approximations for the time-changing domain with clear (local) dependence on the regularity of the boundary values and the boundary itself. The approach is quite general and thus seems suitable for further applications related to the analysis of variable geometries. 
Actually, the present result was already applied, please see the preprints~\cite{Breit,SchSor20,Srdjan}.

Fluid-structure interaction has been an increasingly active area of research in mathematics in the last 20 years. Due to the overwhelming number of contributions in the area we just mention analytic results that are most relevant for our work in this brief literature review. The existence results for weak solutions for the FSI problems where the incompressible Navier-Stokes equations are coupled with a lower-dimensional elasticity model (e.g. plate or shell laws) have been obtained in \cite{CDEM,CG,BorSun,LenRuz,hundertmark2016existence,SunBorMulti}. 
The corresponding existence result for the compressible fluid flow was proved in \cite{BreitSchwarzacher}. All mentioned results on the existence of weak solutions are valid up to time of possible self-intersection of the domain. Up to our knowledge the number of regularity estimates for long time solutions are rather limited. Recently some significant results on strong solutions for large initial data and a $2D$ fluid interacting with a $1D$ solid have been obtained, see~\cite{ grandmont2016existence,grandmont2018existence}. 
For a three dimensional fluid interacting with a three dimensional elastic body see \cite{ignatova2014well,Ignatova17} for the global results with small initial data and structural damping. The theory of local-in-time strong solutions for $3D$-$3D$ FSI problems is rather well developed, see recent results in \cite{boulakia2018well,Kuk,raymond2014fluid} and references within.
We wish to emphasize that in all these works the structure equations were linear. For the FSI problem with non-linear structure the theory is far less developed. The existence of weak solution to the FSI problem with a Koiter membrane energy that includes non-linearities of lower order and a leading order linear regularizing term was proved in \cite{BorSunnon-linear}.
Short time or small data existence result in the context of strong solutions for various non-linear fluid structure models have been obtained in~\cite{ChengShkollerCoutand,ChenShkoller,CSS2,remond2018viscous}. Finally, we wish to mention some results in the static case that can be found here~\cite{GaldiKyed,grandmont2002existence}.

The role of the fluid dissipation on the qualitative properties of the solution is one of the central questions in the area of fluid-structure interactions and related systems, and has been studied by many authors, see e.g.~\cite{AvalosTriggiani2,GaldiZunino,ZhangZuazuaARMA07} and references within. We present here a new technique that allows to transfer dissipation features from the fluid equation to the non-linear hyperbolic elastic displacement. We wish to point out that better regularity can not be expected for a non-linear hyperbolic PDE with arbitrary smooth right hand sides and initial data. It is the coupling with a dissipative equation that allows for this better regularity.

\subsection{The coupled PDE}
We first discuss the Koiter shell model (see e.g. \cite{CiarletVol3,Koiter}) which describes the evolution of the elastic boundary of the fluid domain. Let $\Omega\subset\R^3$ be a domain such that its boundary $\Gamma=\partial\Omega$ is parameterized by a $C^3$ injective mapping $\bvphi:\omega\to\R^3$, where $\omega\subset\R^2$. To simplify notation we assume in this paper that the boundary of $\Omega$ can be parameterized by a flat torus $\omega=\R^2/{\Z^2}$ which corresponds to the assumption of periodic boundary conditions for the structure displacement. We consider the periodic boundary conditions just to avoid unnecessary technical complications (see Remark~\ref{rem:boundary}).
{An example of such domain is a cylindrical domain, see Figure \ref{fig:3d}, with periodic boundary condition for the fluid velocity on the inlet/outlet part of the boundary. Namely, in this setting we can identify the inlet and the outlet part of the boundary, and thus the fluid domain is a $3d$ torus, which fits in our framework.} 

We denote the tangential vectors at any point $\bvphi(\by)$ in the following way:
$$\ba_{\alpha}(\by)=\partial_{\alpha}\bvphi(\by),\; \alpha=1,2,\;\by\in\omega.$$ 
The unit normal vector is given by $\displaystyle{\bn(\by)=\frac{\ba_1(\by)\times\ba_2(\by)}{|\ba_1(\by)\times\ba_2(\by)|}}$. The surface area element of $\partial\Omega$ is given by $dS=|\ba_1(\by)\times\ba_2(\by)|d\by$. We assume that the domain deforms only in the normal direction and denote by $\eta(t,\by)$ the magnitude of the displacement.
This reflects the situation when the fluid pressure is the dominant force acting on the structure in which case it is reasonable to assume that the shell is deforming in normal direction.
In this case the deformed boundary can be parameterized by the following coordinates:
\begin{equation}\label{Parametrization}
\bvphi_{\eta}(t,\by)=\bvphi(\by)+\eta(t,\by)\bn(\by),\; t\in (0,T),\; \by\in\omega.
\end{equation}
We wish to emphasize that this restriction is standard in the majority of mathematical works on the analysis of weak solutions, mainly due to severe technical difficulties associated with the analysis of the case where the full displacement is taken into account. The deformed boundary is denoted by $\Gamma_{\eta}(t)=\bvphi_{\eta}(t,\omega)$. It is a well known fact from differential geometry (see e.g.\ \cite{lee2003introduction}) that there exist $\alpha(\Omega),\beta(\Omega)>0$ such that for $\eta(\by)\in (\alpha(\Omega),\beta(\Omega))$, $\bvphi_{\eta}(t,.)$ is a bijective parameterization of the surface $\Gamma_{\eta}(t)$ and it defines a domain $\Omega_{\eta}(t)$ in its interior such that $\partial\Omega_{\eta}(t)=\Gamma_{\eta}(t)$. Moreover, there exists a bijective transformation $\bpsi_\eta(t,.):\Omega\to\Omega_{\eta}(t)$. For more details on the geometry see Section~\ref{sec:shell} and  Definition~\ref{def:coor}.

We denote the moving domain in the following way:
$$
(0,T)\times\Omega_{\eta}(t):=\bigcup_{t\in (0,T)}\{t\}\times\Omega_{\eta}(t).
$$
The non-linear Koiter model is given in terms of the differences of the first and the second fundamental forms of $\Gamma_{\eta}(t)$ and $\Gamma$ which represent membrane forces and bending forces respectively. These forces are summarized in its potential - the Koiter energy $\Ecal_K(t,\eta)$. The definition of the potential is taken from~\cite[Section 4]{CiarletKoiter}. For a precise definition and the derivation of the energy for our coordinates see~\eqref{KoiterEnergy} below. Let $\mathcal{L}_K\eta$ be the $L^2$-gradient of the Koiter energy $\Ecal_K(t,\eta)$, $h$ be the (constant) thickness of the shell and $\varrho_s$ the (constant) density of the shell. Then the respective momentum equation for the shell reads 
\begin{equation}\label{Structure}
\varrho_s h\partial^2_t\eta+\mathcal{L}_K\eta=g,
\end{equation} 
where $g$ are the momentum forces of the fluid acting on the shell.

\begin{center}
\begin{figure}[h]
\includegraphics[scale=0.35]{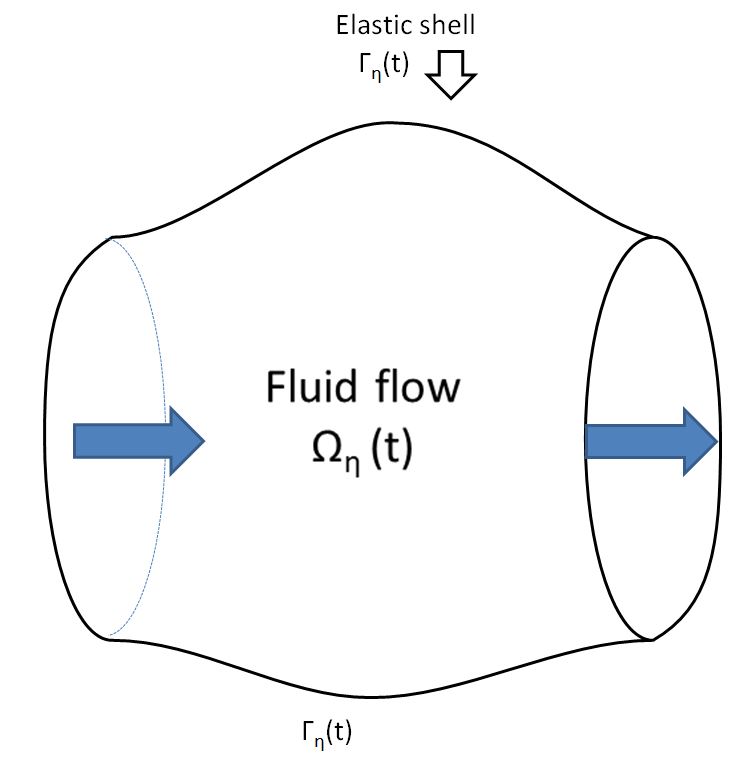}
\caption{An example of the deformed cylindrical domain.}
\label{fig:3d}
\end{figure}
\end{center}

The fluid flow is governed by the incompressible Navier-Stokes equations:
\begin{align}
\varrho_f\big (\partial_t\bu+(\bu\cdot\nabla)\big )\bu&=\divg\bsigma(\bu,p)&\;{\rm in}\;(0,T)\times\Omega_{\eta}(t),
\label{NS}
\\
\divg\bu&=0\;&{\rm in}\;(0,T)\times\Omega_{\eta}(t),
\end{align}
where $\bsigma(\bu,p)=-p\mathbb{I}+2\mu{\rm sym}\nabla\bu$ is the fluid stess tensor
and $\varrho_f$ the (constant) density of the fluid.

The fluid and the structure are coupled via kinematic and dynamic coupling conditions. We prescribe the no-slip kinematic coupling condition which means that the fluid and the structure velocities are equal on the elastic boundary:
\begin{equation}\label{kinematic}
\bu(t,\bvphi_\eta(t,\by))=\partial_t\eta(t,\by)\bn(\by),\quad \by\in \omega.
\end{equation}
The dynamic boundary condition states that the total force in the normal direction on the boundary is zero:
\begin{equation}\label{dynamic}
g(t,y)=-\bsigma(\bu,p)(t,\bvphi_{\eta}(t,\by))\bn(\eta(t,y))\cdot\bn(\by),\quad y\in\omega,
\end{equation}
where $\bn(\eta(t,y))=\partial_1\bvphi_{\eta}(t,y)\times\partial_2\bvphi_{\eta}(t,y)$ is defined as a weighted vector pointing in the direction of the outer normal to the deformed domain at point $\bvphi_{\eta}(t,y)$; the weight is exactly the Jacobian of the change of variables from Eulerian to Lagrangian coordinates. 

We may summarize and state the full fluid-structure interaction problem.
\\
Find $(\bu,\eta)$ such that
\begin{align}
\varrho_f\big (\partial_t\bu+(\bu\cdot\nabla)\big )\bu&=\divg\bsigma(\bu,p)&\;{\rm in}\;(0,T)\times\Omega_{\eta}(t),
\nonumber
\\
\divg\bu&=0\;&{\rm in}\;(0,T)\times\Omega_{\eta}(t),
\nonumber
\\
\varrho_s h\partial^2_t\eta+\mathcal{L}_K\eta&=-(\bsigma(\bu,p)\circ\bvphi_\eta)\bn(\eta)\cdot\bn &{\rm in}\; (0,T)\times\omega,
\label{FSI}
\\
\bu\circ\bvphi_\eta&=\partial_t\eta\bn &{\rm in}\; (0,T)\times \omega,
\nonumber
\\
\bu(0,.)&=\bu_0&{\rm in}\;\Omega_\eta(0),
\nonumber
\\ 
\eta(0)=\eta_0,&\;\partial_t\eta(0)=\eta_1 &{\rm in}\; \omega.
\nonumber
\end{align}
The solution of the above coupled system formally satisfies the following energy equality:
\begin{align}\label{EI}
\begin{aligned}
&\frac{d}{dt}\Big(\frac{\varrho_f}{2}\|\bu(t)\|^2_{L^2(\Omega_{\eta}(t))}
+\frac{h\varrho_s}{2}\|\partial_t\eta(t)\|^2_{L^2(\omega)}
+\mathcal{E}_K(t,\eta)\Big)= -2\mu\int_{\Omega_{\eta}(t)}|{\rm sym}\nabla\bu|^2.
\end{aligned}
\end{align}
Due to the fact that the Koiter shell equation is non-linear---more precisely since the curvature change is measured w.r.t.\ the deformed geometry---the $H^2$-coercivity of the Koiter energy can become degenerate. This is quantified by the estimate that is shown in Lemma~\ref{H2Bound} below. At such degenerate instant the given existence and regularity proofs break down. This is a phenomenon purely due to the non-linearity of the Koiter shell equations. Indeed, in case when the leading order term of the elastic energy is quadratic (i.e. the equation is linear or semi-linear), this loss of coercivity is a-priori excluded.
%
\subsection{Main results}
Let us now state the main theorems of the paper. 
The first main theorem is the existence of solutions to the non-linear Koiter shell model.
\begin{theorem}
\label{thm:ex}
Assume that $\partial \Omega\in C^3, \eta_0\in H^2(\omega), \eta_1\in L^2(\omega)$ and $\bu_0\in L^2(\Omega_{\eta_0})$, and $\eta_0$ is such that $\Gamma_{\eta_0}$ has no self-intersection and $\gamma(\eta_0)\neq 0$. Moreover, we assume that the compatibility condition $\bu_0|_{\Gamma_{\eta_0}}=\eta_1\bn$ is satisfied. Then there exists a weak solution $(\bu,\eta)$ on the time interval $(0,T)$ to \eqref{FSI} in the sense of Definition~\ref{WeakSolDef}.

Furthermore, one of the following is true: either $T=+\infty$, or the structure self-intersect, or $\gamma(\eta)\neq 0$, i.e. the $H^2$-coercivity of the structure energy degenerates, where $\gamma$ is defined in Definition~\ref{def:coor} below.
\end{theorem}

The second main theorem says that all possible solutions in the natural existence class satisfy better structural regularity properties.
\begin{theorem}\label{RegTem}
Let $(\bu,\eta)$ be a weak solution to \eqref{FSI} in the sense of Definition~\ref{WeakSolDef}. Then the solution has the additional regularity property\footnote{For the definition of the fractional Sobolev spaces $H^s(\omega)$ see Subsection~\ref{ssec:notation}} 
$\eta\in L^2(0,T;H^{2+s}(\omega))$ and $\partial_t\eta\in L^2(0,T;H^{s}(\omega))$ for $s\in (0,\frac{1}{2}).$
Moreover, it satisfies the following regularity estimate 
\[
\|\eta\|_{L^2(0,T;H^{2+s}(\omega))}+\|\partial_t\eta\|_{L^2(0,T;H^{s}(\omega))}\leq C_1
\]
with $C_1$ depending on $\partial \Omega$, $C_0$ and the $H^2$-coercivity size $\gamma(\eta)$.
\end{theorem}

\begin{remark}[On the coercivity parameter $\gamma(\eta)$]
Here we briefly discuss  the dependence of constant $C_1$ in Theorem \ref{RegTem} on the coercivity parameter $\gamma(\eta)$. The detailed analysis can be viewed in Section~
\ref{sec:thm2}. First, one can use the energy inequality to deduce that $\eta$ is continuous in space and time, see Lemma~\ref{W14Bound}. The coercivity constant $\gamma(\eta)$ is explicitly calculated in Lemma~\ref{H2Bound}, where it is shown that $\gamma(\eta)$ continuously depends on $\eta$. Consequently (in dependence of the initial state) coercivity in $H^2$ holds at least for some time interval. In particular, on this time interval constant $C_1$ in Theorem~\ref{RegTem} depends only on  $\partial\Omega$ and $C_0$.
If the solution then did not reach a state of degeneracy $(\gamma(\eta)=0$) or self-intersection, then it can be (iteratively) prolonged. In order to derive the estimate for the prolonged solution, one should revise the dependencies in Theorem \ref{RegTem} and this is the reason why in Theorem \ref{RegTem} constant $C_1$ depends on $\gamma(\eta)$. Note in particular, that in some cases $\gamma(\eta)$ can be explicitly computed, see Examples 1 and 2 in the next Section, where very natural conditions for the exclusion of degeneracy are computed.
\end{remark}

\begin{remark}
\label{rem:boundary}
Since here $\partial\Omega$ is assumed to be a manifold without boundary, there is no boundary condition for $\eta$. We restrict ourselves to this case just for technical simplicity. The proof for the case where only a part of the boundary is elastic (see e.g. \cite{BreitSchwarzacher,LenRuz}), with the appropriate Dirichlet boundary conditions, is completely analogous. In particular, the regularity estimate Theorem~\ref{RegTem} is valid for the case of non-periodic shells and the proof can be adapted without any significant complications by using the zero extension of $\eta$ to the whole boundary.
\end{remark}
\begin{remark}
In the proof of Theorem \ref{RegTem} we use the no-slip coupling condition \eqref{kinematic} to transfer the fluid regularity to the structure. Here the exact form of the structure equation~\eqref{Structure} is not essential. Therefore the proof can be easily adapted to different structure models as long as we have that the corresponding structure energy is coercive in the $H^2$ norm. In particular, our result implies respective estimates for weak solutions to several FSI problems that were already studied in the literature and mentioned in the Introduction, e.g. \cite{CDEM,CG,LenRuz,BorSun,BorSunnon-linear}. The geometric condition $\gamma(\eta)\neq 0$ is needed for the $H^2$-coercivity in the non-linear setting. Clearly the $H^2$-coercivity is satisfied independently of $\gamma(\eta)$ when the leading order term in the Koiter energy is quadratic. Hence, in the case when the structure equation is linear or semi-linear the regularity result is valid until a self-intersection is approached.
\end{remark}
\subsection{Novelty \& Significance}

\noindent
The main novelty is the improved regularity of the elastic displacement. In particular it allows to overcome the border between Lipschitz and non-Lipschitz domains. This critical step has caused a significant amount of effort in previous works~\cite{CDEM,CG,LenRuz,BorSun,BorSunnon-linear}. The regularity uses classical differential quotient techniques applied to the non-linear structure equation. Problematic is the impact of the fluid on the structure which is rather implicit in the frame-work of weak equations. Here newly developed extension-operators are developed that are certainly of independent value (see Proposition~\ref{TestFunction}). Of critical technical difficulty are commutator estimates for the time dependent extension of a difference quotient (see Lemma~\ref{lem:IbP}). 

%

The power of the newly introduced method to gain higher regularity for the structure allows to prove the existence of weak solutions for fluids interacting with non-linear Koiter shells. These more physical models have not been in reach for the theory of weak solutions that may exists for arbitrary long times. The mathematical reason is that the respective energies are highly non-linear and non-convex. The extra regularity estimate however, allows to derive sequences that converge strongly to the solution w.r.t\ the highest order of the operator. That is the reason why no {\em linearity or convexity assumptions} are needed anymore to pass to the limit with in the non-linear stress tensor of the structure equation. 

For previous results, the limit passage of the convective term in the Navier Stokes equation was the main effort~\cite{CDEM,CG,LenRuz,BorSun,BorSunnon-linear}. The limit passage usually relies on compactness results of Aubin-Lions type. The variable geometry make its application highly technical. In Section~\ref{sec:auba} we rewrite the celebrated result in a form that we believe to be suitable for coupled systems (see Theorem~\ref{thm:auba}). Indeed, it can be applied to systems where the solution space depends on the solution itself. This section can be seen as the second main technical novelty. 

The interval of existence is potentially arbitrary large. The interval of existence is restricted to cases when the geometry degenerates. However, the minimal interval of existence depends on the reference geometry (which defined the shell model) and can be arbitrary large for some commonly used models. We demonstrate this by providing explicit bounds for two popular reference geometries in Subsection~\ref{sec:shell}; namely the case when the reference geometry is a sphere or a cylinder. 

The method seems very suitable to be adapted for further interaction problems. Possible future applications for FSI problems are in the field of membrane energies, compressible fluids, tangential displacements, uniqueness issues and/or numerical analysis. In two space dimensions or in the regime of low Reynolds numbers the method inherits great potential to further improve the regularity theory for the FSI problems~\cite{grandmont2016existence}. However, due to the lack of the global regularity result for the fluid equations in three dimensions, we do not expect it is possible to further improve the regularity of the shell.
In this sense our regularity result for the shell displacement can be viewed as optimal.


%
This is the point to mention that some applications of our results are already available. Firstly, there is a preprint~\cite{SchSor20} where the additional regularity for weak solution to {\em linear plates} is used in order to obtain {\em weak-strong uniqueness} results. In the case of $3D$ fluids, the proof depends crucially on the extra regularity obtained here. Secondly, the results where used in a recent preprint~\cite{Breit} where Koiter shells are coupled to polymer fluids. Thirdly, we wish to mention a work in progress on heat-conducting fluids~\cite{BreSch20} where the Navier-Stokes-Fourier system is considered. In order to obtain an energy equality, that is an essential part of the concept of weak solutions, compactness of the elastic solid energy is crucial.  Hence, this result also depends sensitively on the regularity techniques presented here. Finally, our methodology for the regularity estimate was also used in the proof of a weak-strong uniqueness result for FSI problems with compressible fluid \cite{Srdjan}.

{\bf Outline of the paper:}
The next section first derives the Koiter energy w.r.t.\ our chosen coordinates, gives two explicit examples of Koiter energies with respective geometric restrictions on $
\alpha(\Omega),\beta(\Omega),
\gamma(\eta)$, and introduces the definition of a weak solution for fluid-structure interactions. Section~\ref{sec:Test} is the technical heart of the paper since there the solenoidal extension and approximation operators are introduced.
 In Section~\ref{sec:thm2}, we give the proof of the regularity result Theorem~\ref{RegTem}. Section~\ref{sec:auba} provides a new version of Aubin-Lions compactness result that reveals the connection between the existence of suitable extensions and compactness results for fluid-structure interactions involving elastic shells. Finally, in Section~\ref{sec:thm1} we show Theorem~\ref{thm:ex}; the existence is shown by combining the extra regularity of the shell with the compactness theory.

\section{Weak solutions}
\subsection{The elastic energy}
\label{sec:shell}

\noindent
{\bf Coordinates.}

Here we follow the strategy of~\cite[Section 2]{LenRuz} by introducing the following coordinates {attached to the reference geometry $\Omega$ which are well defined in the tubular neighborhood of $\partial\Omega$ (see e.g.~\cite[Section 10]{lee2003introduction} and Figure~\ref{CoordFig} for an illustration). 
\begin{figure}[h]
\includegraphics[width=\linewidth]{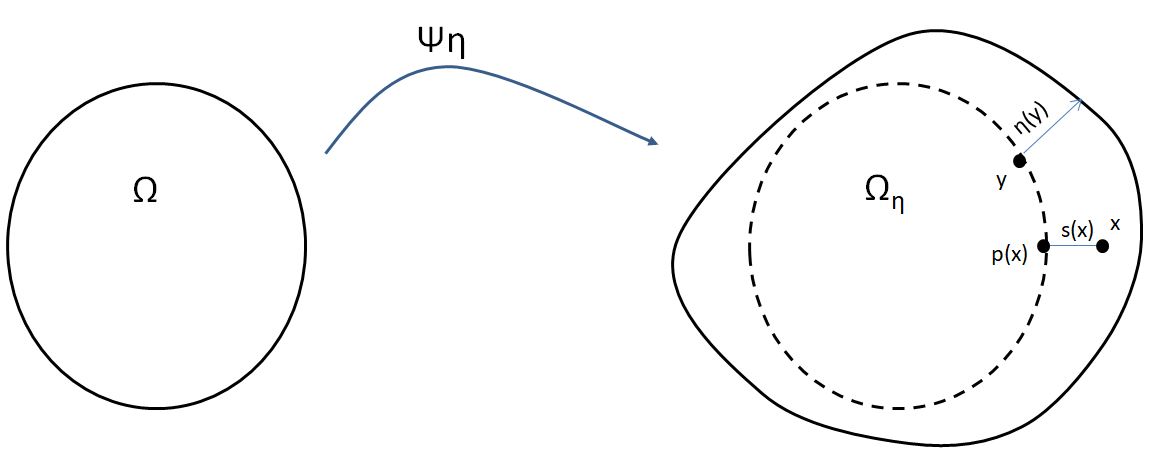}
\caption{Cross section of cylindrical domain, its deformation and corresponding coordinate system.}
\label{CoordFig}
\end{figure}
\begin{definition}
\label{def:coor}
Let $x$ be a point in the neighborhood of $\partial\Omega$. We define the distance parameter with respect to the reference point
\[
 \by(x)=\text{ arg min}_{\by\in \omega}\abs{x-\bvphi(\by)}, \quad s(x)=(x-y(x))\cdot \bn(x)
\]
and the projection $\bfp(x)=\bvphi(\by(x))$. 


We define the numbers $\alpha(\Omega),\beta(\Omega)$ so that $(\alpha(\Omega),\beta(\Omega))$ is the largest open interval such that numbers $s(x),\bfp(x),\by(x)$ are uniquely defined over 
$ \{\bvphi(y)+s\bn:y\in\omega,\; s\in (\alpha(\Omega),\beta(\Omega))\}$.

For $\kappa>0$ we introduce the indicator mapping $\sigma_\kappa\in C^\infty(\R)$, such that
\[
\sigma_\kappa(s)=1\text{ for } s\in (\alpha(\Omega)+\kappa,\infty),\,\sigma_\kappa(s)=0\text{ for }s\leq\alpha(\Omega)+\frac{\kappa}2\text{ and }\sigma_\kappa'\geq 0.
\] 
We set
\[
S_{\kappa}=
\{\bvphi(y)+s\bn(y):(s,y)\in [\alpha(\Omega)+\kappa,\beta(\Omega)-\kappa]\times \omega\}
\]
Further
\[
 Q^\kappa=S_\kappa\cup \Omega\text{ and }Q_\kappa=\Omega\setminus S_\kappa. 
\]
 In particular we have a clear information on the support of the derivative:
\[\Acal_\kappa:=
\{\bvphi(y)+s\bn(y):(s,y)\in [\alpha(\Omega)+\kappa/2,\alpha(\Omega)+\kappa]\times \omega\}\supset\mathrm{supp}(\sigma_\kappa'(s(\cdot))).
\] 
For $\eta(\by)\in  (\alpha(\Omega)+\kappa,\beta(\Omega)-\kappa)$, this allows to introduce the mapping $\bpsi_\eta(t,.):\Omega\to\Omega_{\eta(t)}$ by
\[
z\mapsto \Big (1-\sigma_\kappa(s(z))\Big )z+\sigma_\kappa(s(z))\Big (\bfp(z)+\big (\eta(\by(z))+s(z)\big )\bn(\by(z))\big ),
\]
and $\bpsi_\eta^{-1}(t,.):\Omega_{\eta(t)}\to\Omega$ by 
\[
x\mapsto \Big (1-\sigma_\kappa(s(x))\Big )x+\sigma_\kappa(s(x))\Big (\bfp(x)+\big (s(x)-\eta(\by(x))\big )\bn(\by(x))\Big ).
\]

Moreover we define 
\[
\bfPhi : (\alpha(\Omega)+\kappa,\beta(\Omega)-\kappa)\times \omega \to S_\kappa,\quad \bfPhi(s,y)=\bvphi(y)+\bn(y)s.
\]
The function is smooth and invertible in dependence of $\bvphi$ and $\kappa$.
This implies that 
\[
\psi_\eta\circ\bfPhi : (\alpha(\Omega)+\kappa,0)\times \omega \to \Omega_\eta,
\quad  \psi_\eta\circ\bfPhi(s,y)= \bvphi(y)+(s+\sigma_\kappa(s)\eta(y))\bn(y)
\]

Finally, we define the following geometric quantity depending on $\partial\Omega$ and $\eta$:
\begin{equation}\label{GammaDef}
\gamma(\eta)=\frac{1}{|a_1\times a_2|}\Big (|a_1\times a_2|+
\eta(\bn\cdot(\ba_1\times\partial_2\bn+\partial_1\bn\times\ba_2))
+\eta^2\bn\cdot(\partial_1\bn\times\partial_2\bn)\Big ).
\end{equation}
\end{definition}

\begin{remark}
The numbers $\alpha(\Omega),\beta(\Omega)$ do not have to be small. For example, if $\Omega$ is a ball or a cylinder with radius $R$, then $(\alpha(\Omega),\beta(\Omega))=(-R,\infty)$. {The geometric quantity $\gamma(\eta)$ is connected to the $H^2$-coercivity of the non-linear structure energy and its meaning is clarified in Lemma~\ref{H2Bound} and Remark~\ref{rem:T}.}
\end{remark}

\noindent
{\bf Derivation of the elastic energy.}

The non-linear Koiter model is given in terms of the differences of the first and the second fundamental forms of $\Gamma_{\eta}(t)$ and $\Gamma$. The tangent vectors to the deformed boundary are given by:
\begin{equation}\label{DeformedTangent}
\ba_{\alpha}(\eta)=\partial_{\alpha}\bvphi_{\eta}=\ba_{\alpha}+\partial_{\alpha}\eta\bn+\eta\partial_{\alpha}\bn,\; \alpha=1,2.
\end{equation}
Therefore, the components of the first fundamental form of the deformed configuration are given by:
\begin{equation}\label{1Deformed}
a_{\alpha\beta}(\eta)=\ba_{\alpha}(\eta)\cdot\ba_{\beta}(\eta)
=a_{\alpha\beta}+\partial_{\alpha}\eta\partial_{\beta}\eta
+\eta\big (\ba_{\alpha}\cdot\partial_{\beta}\bn+\ba_{\beta}\cdot\partial_{\alpha}\bn\big )
+\eta^2\partial_{\alpha}\bn\cdot\partial_{\beta}\bn.
\end{equation}
We define the change of metric tensor $\bG(\eta)$:
\begin{equation}\label{ChangeMetric}
G_{\alpha\beta}(\eta)=a_{\alpha\beta}(\eta)-a_{\alpha\beta}
=\partial_{\alpha}\eta\partial_{\beta}\eta
+\eta\big (\ba_{\alpha}\cdot\partial_{\beta}\bn+\ba_{\beta}\cdot\partial_{\alpha}\bn\big )
+\eta^2\partial_{\alpha}\bn\cdot\partial_{\beta}\bn.
\end{equation}
The normal vector to the deformed configuration is given by:
\begin{equation}\label{NormalDeformed}
\begin{array}{c}
\bn(\eta)=\ba_1(\eta)\times \ba_2(\eta)
=|\ba_1\times \ba_2|\bn+\partial_2\eta(\ba_1\times\bn+\eta\partial_1\bn\times\bn)
\\ \\
+\partial_1\eta(\bn\times\ba_2+\eta\bn\times\partial_2\bn)+ \eta(\ba_1\times\partial_2\bn+\partial_1\bn\times\ba_2)
+\eta^2(\partial_1\bn\times\partial_2\bn).
\end{array}
\end{equation}
Notice that $\bn(\eta)$ is not a unit vector. We follow our reference literature~\cite{CiarletKoiter} and use the following tensor $\bR$ (denoted by $\bR^{\#}$ in \cite[Section 4]{CiarletKoiter}) which is some non-normalized variant of the second fundamental form to measure the change of curvature:
\begin{equation}\label{CurvatureChange}
R_{\alpha\beta}(\eta)=\frac{1}{|\ba_1\times\ba_2|}\partial_{\alpha}\ba_{\beta}(\eta)\cdot\bn(\eta)
-\partial_{\alpha}\ba_{\beta}\cdot\bn,\; \alpha,\beta=1,2.
\end{equation}
Finally, we define the elasticity tensor in the classical way~\cite[Theorem 3.2]{CiarletKoiter}:
\begin{equation}\label{ElasticityTensor}
\mathcal{A}\mathbf{E}=\frac{4\lambda\mu}{\lambda+2\mu}(\mathbf{A}: \mathbf{E})\mathbf{A}+4\mu\mathbf{A}\mathbf{E}\mathbf{A},\quad \mathbf{E}\in {\rm Sym}(\R^{2\times 2}).
\end{equation}
Here $\mathbf{A}$ is the contravariant metric tensor associated with $\partial\Omega$ ({see e.g. \cite[Section 2]{CiarletKoiter} for the precise definition of $\mathbf{A}$}), and $\lambda>0$, $\mu>0$ are the Lam\' e constants. The Koiter energy of the shell is given by:
\begin{equation}\label{KoiterEnergy}
\mathcal{E}_{K}(t,\eta)=\frac{h}{4}\int_{\omega}\mathcal{A}\mathbf{G}(\eta(t,.)):\mathbf{G}(\eta(t,.))dy
+\frac{h^3}{48}\int_{\omega}\mathcal{A}\bR(\eta(t,.)):\bR(\eta(t,.))dy,
\end{equation}
where $h$ is the thickness of the shell. In order to simplify the notation we introduce the following forms connected to the membrane and bending effects in the variational formulation:
\begin{align}\label{ElasticForms}
a_m(t,\eta,\xi)=\frac{h}{2}\int_{\omega}\mathcal{A}\bG(\eta(t,.)): \bG'(\eta(t,.))\xi\, dy,
\\
a_b(t,\eta,\xi)=\frac{h^3}{24}\int_{\omega}\mathcal{A}\bR(\eta(t,.)): \bR'(\eta(t,.))\xi\, dy,
\end{align}
where $\bG'$ and $\bR'$ denote the Fr\' echet derivatives of $\bG$ and $\bR$ respectively. Therefore, the elastodynamics of the shell is given by the following variational formulation:
\begin{equation}
h\varrho_s\frac{d}{dt}\int_{\omega}\partial_t\eta(t,.)\xi\, dy+
a_m(t,\eta,\xi)+a_b(t,\eta,\xi)=\int_{\omega}g\xi\, dy\; {\rm on}\; (0,T),\; \xi\in W^{2,p}(\omega), 
\end{equation}
where $\varrho_s$ is the structure density, $g$ is the density of area force acting on the structure, and $p>2$. We denote the elasticity operator by $\mathcal{L}_K$ which is formally given by
\begin{equation}\label{ElasticityOp} 
\langle \mathcal{L}_K\eta,\xi\rangle = a_m(t,\eta,\xi)+a_b(t,\eta,\xi),\quad \xi\in W^{2,p}(\omega).
\end{equation}
Next we give some examples for which we can calculate our restrictive numbers $\alpha(\Omega),\beta(\Omega)$ and $\gamma(\eta)$.

\smallskip
{\bf Example 1: Cylindrical Koiter shell}

\noindent
The parameterization of the reference cylinder is given by $\bvphi(\theta,z)=(R\cos\theta,R\sin\theta,z)$, $(\theta,z)\in \omega=(0,2\pi)\times (0,1)$, where $R>0$ is the radius of the cylinder. We compute
$$
\ba_1(\theta,z)=(-R\sin\theta,R\cos\theta,0),\;
\ba_2(\theta,z)=(0,0,1),\;
\bn(\theta,z)=(\cos\theta,\sin\theta,0).
$$ 
The corresponding contravariant metric tensor is given by 
$A=\left (\begin{array}{cc}\frac{1}{R^2} & 0 \\ 0 & 0\end{array} \right ).$
The deformation of the cylindrical boundary is given by:
$$
\bvphi_{\eta}(\theta,z)=(R\cos\theta+\eta(\theta,z)\cos\theta,R\sin\theta+\eta(\theta,z)\sin\theta,z).
$$
Straightforward calculation yields:
$$
\ba_1(\eta)=(1+\frac{1}{R})\ba_1+\eta_{\theta}\bn,\;
\ba_2(\eta)=\ba_2+\eta_z \ba_3,
$$
$$
\bn(\eta)=(R+\eta)\bn-\eta_z(R+\eta)\ba_2+\frac{\eta_{\theta}}{R}\ba_1.
$$
Therefore, the change of metric tensor is given by
$$
\bG(\eta)=\left (\begin{array}{cc}(R+\eta)^2+\eta_{\theta}^2-R^2 & \eta_{\theta}\eta_z \\ \eta_{\theta}\eta_z & 1+\eta_z^2\end{array} \right ),
$$
and the change of curvature tensor by
$$
\bR(\eta)=\left (\begin{array}{cc}(1+\frac{\eta}{R})\eta_{\theta\theta}-\frac{1}{R}(\eta+R)^2-2\frac{\eta_\theta^2}{R}+R &(1+\frac{\eta}{R})\eta_{\theta z}-\frac{1}{R}\eta_{\theta}\eta_z \\ (1+\frac{\eta}{R})\eta_{\theta z}-\frac{1}{R}\eta_{\theta}\eta_z & (1+\frac{\eta}{R})\eta_{zz} \end{array} \right ).
$$
Here $(\alpha(\Omega),\beta(\Omega))=(-R,\infty)$ and $\gamma(\eta)=1+\frac{\eta}{R}$.

\smallskip
{\bf Example 2: Spherical shell}

\noindent
Strictly speaking, the sphere does not fit in our framework since it does not have a global parameterization. However, this assumption was introduced just for technical simplicity and can be easily removed by working with local coordinates. In this example we consider an elastic sphere with holes around north and south poles. On these holes we prescribe the boundary condition for the fluid flow, e.g. inflow/outflow or Dirichlet. The shell is clamped on the boundary of the holes (see Figure~\ref{fig:3d} for illustration}). More precisely, the parameterization is given by
$$
\bvphi(\theta,\phi)=R(\cos\theta\sin\phi,\sin\theta\sin\phi,\cos\phi),\; (\theta,\phi)\in w=(0,2\pi)\times (a,\pi-a),
$$
where $R>0$ is the radius of the sphere, and $a>0$ is the parameter determining the size of the holes. We compute the tangent and normal vectors to the reference configuration
$$
\ba_1=-R(\sin\theta\sin\phi,\cos\theta\sin\phi,0),\;
\ba_2=R(\cos\theta\cos\phi,\sin\theta\cos\phi,-\sin\pi),\;
$$
$$
\bn=-(\cos\theta\sin\phi,\sin\theta\sin\phi,\cos\phi).
$$
The contravariant metric tensor is given by $A=\left (\begin{array}{cc}\frac{1}{R^2\sin^2\phi} & 0 \\ 0 & \frac{1}{R^2}\end{array} \right ),$ and the deformation of the cylindrical boundary by
$$
\bvphi_\eta(\theta,\phi)=(R-\eta(\theta,\phi))(\cos\theta\sin\phi,\sin\theta\sin\phi,\cos\phi).
$$
We calculate the tangent and normal vectors to the deformed configuration:
$$
\ba_1(\eta)=(1-\frac{\eta}{R})\ba_1+\eta_{\theta}\bn,\;
\ba_2(\eta)=(1-\frac{\eta}{R})\ba_2+\eta_{\phi}\bn,
$$
$$
\bn(\eta)=(R-\eta)^2 \sin\phi\bn-(1-\frac{\eta}{R})\Big (\frac{\eta_{\theta}}{\sin\phi} \ba_1+\eta_{\phi}\sin\phi\ba_2 \Big ).
$$
The change of the metric tensor is given by
$$
\bG(\eta)=\left (\begin{array}{cc}\eta_\theta^2+(\sin\phi)^2\eta(\eta-2R) & \eta_{\theta}\eta_\phi \\ \eta_{\theta}\eta_\phi & \eta_{\phi}^2+(R-\eta)^2-R^2\end{array} \right ).
$$
Finally, the components of the change of curvature tensor are given by
\begin{align*}
R_{11}(\eta)&=\frac{1}{2R^2}\Big (-2\eta^3\sin^2\phi+\eta^2(6R\sin^2\phi+\eta_{\phi}\sin 2\phi+2\eta_{\theta\theta})
\\
&\quad -2\eta (3R^2\sin^2\phi+R\eta_{\phi}\sin 2\phi+2\eta_{\theta}^2+2R\eta_{\theta\theta})+R(R\eta_{\phi}\sin 2\phi+4\eta_{\theta}^2+2R\eta_{\theta\theta})\Big )
\end{align*}
$$
R_{12}(\eta)=R_{21}(\eta)=\frac{\eta-R}{R^2}\Big (
\eta_{\theta}(R\cot\phi-\eta\cot\phi-2\eta_{\phi})
+\eta_{\theta\phi}(\eta-R)\Big )
$$
$$
R_{22}(\eta)=\frac{1}{R^2}\Big (-\eta^3+\eta^2 (3R+\eta_{\phi\phi})+R(2\eta^2_{\phi}+R\eta_{\phi\phi})-\eta(2\eta_{\phi}^2+R(3R+2\eta_{\phi\phi}))\Big ).
$$
The clamped boundary conditions are $\eta=\partial_{\phi}\eta=0$, $\phi=a,\pi-a$. Since we will take finite differences of order less than $1$, we can extend $\eta$ by zero (over the poles) and complete all estimates related to the regularity.
Here $(\alpha(\Omega),\beta(\Omega))=(-\infty,R)$ and $\gamma(\eta)=\frac{(\eta-R)^2}{R^2}$.

\subsection{Weak coupled solutions}
We use here the standard notation of Bochner spaces related to Lebesgue and Sobolev spaces. We will use bold letters for vector valued functions in three dimensions. Usually we take $y\in \omega$ to be a two dimensional variable and $x$ as a three dimensional variable.
In order to define weak solutions, let us first define the appropriate function spaces:
\begin{align}
\label{SolSpace}
\begin{aligned}
&\mathcal{V}_\eta(t)=\{\bu\in H^1(\Omega_{\eta}(t)):\divg\bu=0\},
\\
&\mathcal{V}_F=L^{\infty}(0,T;L^2(\Omega_\eta(t))\cap L^2(0,T;V_{\eta}(t)),
\\
&\mathcal{V}_K=L^{\infty}(0,T;H^2(\omega))\cap W^{1,\infty}(0,T;L^2(\omega)),
\\
&\mathcal{V}_S=\{(\bu,\eta)\in\mathcal{V}_F\times\mathcal{V}_K:\bu(t,\bvphi_\eta(t,.))=\partial_t\eta(t,.)\bn(\eta(t,.))\},%
\\
&\mathcal{V}_T=\{(\bq,\xi)\in\mathcal{V}_F\times\mathcal{V}_K:\bq(t,\bvphi_\eta(t,.))=\xi(t,.)\bn(\eta(t,.)),\; \partial_t\bq\in L^2(0,T;L^2(\Omega_{\eta}(t))\}.
\end{aligned}
\end{align}
Here $\mathcal{V}_S$ and $\mathcal{V}_F$ are solution and test space respectively.
Even though for $\eta\in\mathcal{V}_K$, $\Omega_{\eta}(t)$ is not necessary a Lipschitz domain, the traces used in definitions \eqref{SolSpace} and \eqref{TestFunction} are well defined, see Corollary 2.9. from \cite{LenRuz} (see also \cite{CDEM,BorisTrag}).
We introduce the concept of solution which we will consider here. Observe, that from this point on we normalize all physical constants $\rho_s=\rho_f=h=\mu=\lambda=1$ for notational simplicity since the proofs require just positivity of these constants. We emphasize that the restrictions on existence and regularity are only of geometrical nature. This can be quantified by $\alpha(\Omega)$ and $\beta(\Omega)$ depending only on the reference geometry, and $\gamma(\eta)$ depending on the reference geometry and on the particular magnitude and direction of the displacement, but not on the above physical constants. 
\begin{definition}[Weak solution]\label{WeakSolDef} We call $(\bu,\eta)\in \mathcal{V}_S$ a weak solution of problem \eqref{FSI} if it satisfies the energy inequality \eqref{EI-weak} and for every $(\bq,\xi)\in\mathcal{V}_T$ the following equality holds in $\mathcal{D}'(0,T)$
\begin{equation}\label{WeakSolEq}
\begin{array}{c}
\displaystyle{\frac{d}{dt}\int_{\Omega_{\eta}(t)}\bu\cdot\bq\, dx+\int_{\Omega_{\eta}(t)}\Big (-\bu\cdot\partial_t\bq-\bu\otimes\bu:\nabla\bq+2\sym\nabla\bu:\sym\nabla\bq\Big )\, dx}
\\
\displaystyle{+\frac{d}{dt}\int_{\omega}\partial_t\eta\xi\, dy-\int_{\omega}\partial_t\eta\partial_t\xi+a_m(t,\eta,\xi)+a_b(t,\eta,\xi)\,d\by=0,}
\end{array}
\end{equation}
Furthermore, the initial values $\eta_0,\eta_1,\bu_0$ are attained in the respective weakly continuous sense.
\end{definition}
By formally multiplying \eqref{FSI}$_1$ by $\bu$ and \eqref{FSI}$_2$ by $\partial_t\eta$,  integrating over $\Omega_{\eta}(t)$ and $\omega$ respectively, integrating by parts and using the coupling conditions \eqref{FSI}$_4$, we obtain the energy inequality (see e.g. \cite{CDEM,BorSun} for details of the computations related to the change of the domain and the convective term):
\begin{align}\label{EI-weak}
\begin{aligned}
&\frac{1}{2}\|\bu(t)\|^2_{L^2(\Omega_{\eta}(t))}
+\frac{1}{2}\|\partial_t\eta(t)\|^2_{L^2(\omega)}
+\mathcal{E}_K(t,\eta)+2\int_0^t\int_{(\Omega_{\eta}(t))}|{\rm sym}\nabla\bu|^2\, dx\, dt 
\\
&\quad \leq 
\frac{1}{2}\|\bu_0\|^2_{L^2(\Omega_{\eta_0})}
+\frac{1}{2}\|\eta_1\|^2_{L^2(\omega)}
+\mathcal{E}_K(0,\eta_0)=:C_0.
\end{aligned}
\end{align}

\subsection{Fractional spaces}
\label{ssec:notation}
In the paper, we use the standard definitions of Bochner spaces related to Lebesgue and Sobolev spaces. In particular, we consider {\em fractional Sobolev spaces} and {\em Nikolskii spaces}. We recall their definitions here.

For $\alpha\in (0,1)$ (the order of derivative) and $q\in [1,\infty)$ (the exponent of integrability) we say that $g\in W^{\alpha,q}(A)$, for a domain $A\subset \Rbb^d$ if its norm
\[ \norm{g}_{W^{\alpha,q}(A)}^q:=\bigg(\int_{A}\int_A\frac{\abs{g(x)-g(y)}^q}{\abs{x-y}^{n+\alpha q}} dx\, dy\bigg)^\frac1q+\bigg(\int_A\abs{g(x)}^q dx \bigg)^\frac1q
\]
is finite. Fractional Sobolev spaces can be extended to higher order. For $\alpha\in [k,k+1)$ with $k\in \Nbb$ it is said that $g\in W^{\alpha,q}(A)$, if all partial derviatives of order up to $k$ are in $W^{\alpha-k,q}(A)$.
In the particular case $q=2$ we use the abbreviation
\[
H^{s}(A)\equiv W^{s,2}(A)\text{ for }s\in [0,\infty). 
\]
We say that $g\in N^{\alpha,q}(A)$ if its norm
\[ \norm{g}_{N^{\alpha,q}(A)}:=\sup_{i\in\{1,...,d\}}\sup_{h\neq 0}\bigg(\int_{A_h}\Big |\frac{g(x+he_i)-g(x)}{\abs{h}^{\alpha-1}h}\Big|^q dx\bigg)^\frac1q+\bigg(\int_A\abs{g(x)}^q dx\bigg)^\frac1q ,
\]
where $e_i$ is the $i$-th unit vector and $A_h=\{x\in A:\text{ dist}(x,\partial A)>h\}$, is finite.
Nikolskii spaces are closely related to fractional Sobolev spaces $W^{\alpha,q}(A)$.  
Let us just mention that for $0<\alpha<\beta<1$ and a bounded domain $A$ we have
\[
W^{\beta,q}(A)\subset N^{\beta,q}(A)\subset 
W^{\alpha,q}(A).
\]
Recall also that for fractional Sobolev spaces an embedding theorem is available for a Lipschitz domain $A\subset\Rbb^n$ and $g\in N^{\beta,q}(A)$ and $0<\alpha<\beta<1$ we have for $\alpha q<n$ that
\begin{align}
\label{eq:se}
 \norm{g}_{L^\frac{nq}{n-\alpha q}(A)}\leq c_1\norm{g}_{W^{\alpha,q}(A)}\leq c_2\norm{g}_{N^{\beta,q}(A)},
\end{align}
and for $\alpha q>n$
\begin{align}
\label{eq:se2}
 \norm{g}_{C^{\alpha-\frac{n}{q}}(A)}\leq c_1\norm{g}_{W^{\alpha,q}(A)}\leq c_2\norm{g}_{N^{\beta,q}(A)}.
\end{align}
 For the above estimates and more detailed study on the given function spaces we refer to~\cite[Chapter 7]{Ada75} and \cite{Tri92}.
 The Nikolskii spaces are very popular in the analysis of PDE, since their definition via difference quotients is rather easy to handle. Namely we introduce for $g\in L^1(\omega)$ and $h\neq 0$
 \[
 D^{s}_{h,e}(g)(x):=\frac{g(x+he)-g(x)}{\abs{h}^{s-1}h}\text{ for any (unit) vector }e\in \Rbb^2.
 \]
In the following, we will omit mentioning the direction $e$ since it is never of relevance and write
$  D^{s}_{h}(q)(y):= D^{s}_{h,e}(q)(y)$ for an arbitrary direction $e$. At this point we just wish to mention that these expressions satisfy the following summation-by-parts formula 
\begin{align*}
\int_\omega  D^{s}_{h,e}(g)(y)q(y)\, dy=-\int_\omega  g(y) D^{s}_{-h,e}(q)(y)\, dy
\end{align*}
for all periodic functions $g\in L^p(\omega)$ and $q\in L^{p'}(\omega)$ with $p\in [1,\infty]$.

For a vector field $\bfg:A\to \Rbb^d$, we say that $\bfg\in W^{\alpha,p}(A)$, if $\bfg^i\in W^{\alpha,p}(A)$ for all $i\in \{1,...,d\}$. Finally we denote by
\[
W^{\alpha,p}_{\diverg}(A)=\{\mathbf{g}\in W^{\alpha,p}_{\diverg}(A)\,:\, \diverg(\bfg)=0\text{ in a distributional sense}\}.
\]

\section{Solenoidal extensions and smooth approximations}\label{sec:Test}
In this section we construct a divergence free extension operator from $(0,T)\times\partial\Omega$ to $(0,T)\times\Omega_{\eta}(t)$. The construction is based on the ideas of the construction in~\cite[Prop. 2.11]{LenRuz}.
In contrast to the approach there we will use the celebrated \Bogovskij{} theorem in place of the steady Stokes operator. We use the following theorem that can be found in~\cite[Section 3.3]{galdi2011introduction}, and in \cite[Appendix 10.5]{FeireislNovotny09}. 
\begin{theorem}\label{BogovskiiTeorem1}
Let $\Omega$ be uniformly Lipschitz.
There exists a linear operator $\bog:\hat{C}^{\infty}_{0}(\Omega)\to C^{\infty}_{0}(\Omega)^d$, {with the property $\divg(Bog(f))=f$}, which extends from $\hat{W}_0^{k-1,p}(\Omega)\to W^{k,p}_0(\Omega)$ for $1<p<\infty$ and $k\in \{0,1,2,...\}$, such that
\begin{align}\label{BogEstimates}
\|\bog(f)\|_{W^{k,p}(\Omega)}\leq C \|f\|_{\hat{W}^{k-1,p}(\Omega)},\quad k\in \Z,
\end{align}
where $C$ is an absolute constant depending only on the Lipschitz constant. 
Here we use the notation $\hat{C}^{\infty}_{0}(\Omega)=\{f\in C^{\infty}_0(\Omega):\int_{\Omega}fdx=0\}$, and for $l\geq 0$, $\hat{W}_0^{l,p}(\Omega)=\{f\in W_0^{l,p}(\Omega):\int_{\Omega}f dx=0\}$,  
$\hat{W}^{-l,p}_0(\Omega)=\{ f\in \mathring{W}^{-l,p}(\Omega)\, :\, \skp{f,1}=0\}$, where $\mathring{W}^{-l,p}(\Omega)$ is defined via the norm
$$
\|f\|_{\mathring{W}^{-l,p}(\Omega)}=\sup_{\{\phi\in W^{l,p'}(\Omega):\|\phi\|_{W^{l,p'}}=1\}}\skp{f,\phi}.
$$
\end{theorem}

\noindent
Within this section we assume that $\eta:[0,T]\times \omega\to \Rbb$ is such that there exists $\alpha_{\eta}$, $\beta_{\eta}$ such that
\begin{align}
\label{eq:bounds}
\text{  $\alpha(\Omega)+\kappa\leq\alpha_\eta\leq  \eta(t,\by)\leq \beta_\eta\leq \beta(\Omega)-\kappa$ for all $(t,\by)\in [0,T]\times \omega$.}
\end{align} 
Moreover, in this section we use $c$ or $C$ as generic constants which may change their sizes in different instances. Since their  dependence on the geometry is relevant for our arguments, it will always be given explicitly in the statements of the results.
 
The first step is to introduce a solenoidal extension operator. However, since all functions defined on the boundary do not necessary allow for a solenoidal extension, we first need to construct a suitable corrector.
We use the coordinates introduced in Definition~\ref{def:coor}. For ease of readability we define for a function $\xi:\omega \to\R$ 
\[
\tilde{\xi}:\partial\Omega\to \R\text{ with }\tilde{\xi}(\bfp(x))=\tilde{\xi}(x):=\xi(\by(x)).
\]
In our solenoidal extension the \Bogovskij{} theorem will be applied to
\[
 S_\frac{\kappa}{2}\setminus S_\kappa=: \Acal_\kappa.
\]
Observe, that $\Acal_\kappa$ is a $C^2$ domain that contains the support of the function $(\bfp,s)\mapsto {\sigma_\kappa'(\bfp+s\tilde{\bn}(\bfp))}$ from Definition~\ref{def:coor}.

Next we introduce the following weighted mean-value over that set. Let $\lambda\in L^\infty( {\Acal_\kappa}), \lambda \geq 0$, and $\int_{ {\Acal_\kappa}} \lambda(x)\, dx>0$ be a given weight. Then
\[
\mean{\psi}_{\lambda}:=
\frac{\int_{ {\Acal_\kappa}} \psi (x) \lambda(x)\, dx}{\int_{ {\Acal_\kappa}} \lambda(x)\, dx}
\text{ for }\psi\in L^1( {\Acal_\kappa}).
\]
We will denote
\begin{align}
\label{eq:weight}
{{\lambda_{\eta}}}({t},x):=e^{(s(x)-\eta(t,\by(x)))\divg(\bn(\bfp(x)))}\sigma_\kappa'(s(x))\geq 0,
\end{align}
which has compact support in $\Acal_\kappa$ and satisfies (uniformly in $t$)
\[
c_1\leq\norm{{{\lambda_{\eta}}}}_{L^1(\Acal_\kappa)}\leq c_2\norm{{{\lambda_{\eta}}}}_{L^\infty(\Acal_\kappa)}\leq c_3
\]
for some positive constants $c_1\leq c_2\leq c_3$ depending just on $\kappa$ and the upper and lower bounds of $\eta$.
%

\begin{corollary}[Corrector]
\label{Corrector}
\label{cor:cor}
Let \eqref{eq:bounds} be satisfied. Then the corrector map  
\[
\cor:L^1(\omega)\to \Rbb,\, \quad\xi\mapsto\cor(\xi)=\mean{\tilde{\xi}}_{{\lambda_{\eta}}}=\frac{\int_{ {\Acal_\kappa}} \tilde{\xi}(\bfp(x)) {{\lambda_{\eta}}}({t},x)\, dx}{\int_{ {\Acal_\kappa}} {{\lambda_{\eta}}}({t},x)\, dx},
\]
satisfies the following estimates for $q\in [1,\infty]$:
\begin{align}
\|\cor(\xi)\|_{L^{q}(0,T)}&\leq C \|\xi\|_{L^{q}(0,T;L^1(\omega))},
\label{CorBound}
\\
\|\partial_t\cor(\xi)\|_{L^{q}(0,T)}
&\leq C
\Big(\|\partial_t \xi\|_{L^{q}(0,T;L^1(\omega))}
+\| \xi \partial_t\eta\|_{L^{q}(0,T;L^1(\omega))}\Big),
\label{CorTimeDBound}
\end{align}
whenever the right hand side is finite. Here $C$ depends only on $\alpha_\eta, \beta_\eta,$ and $\kappa$.

\end{corollary}
\begin{proof}
The estimates in $L^q(0,T)$ are immediate by the uniform bounds of $\lambda_\eta$ and $\sigma$. In order to estimate the time-derivative, we use the calculation 
\begin{align*}
\partial_t\mean{\xi(t)}_{{{\lambda_{\eta}}}(t)}&=-\frac{1}{\norm{{{\lambda_{\eta}}}(t)}_{L^1}^2}\int_{ \Acal_\kappa}\partial_t{{\lambda_{\eta}}}(t)\, dx\int_{ \Acal_\kappa}\tilde{\xi}(t){{\lambda_{\eta}}}(t)\, dx+\frac{1}{\norm{{{\lambda_{\eta}}}(t)}_{L^1}}\int_{ \Acal_\kappa}\partial_t\tilde{\xi}(t){{\lambda_{\eta}}}(t)\, dx 
\\
&\quad + \frac{1}{\norm{{{\lambda_{\eta}}}(t)}_{L^1}}\int_{ \Acal_\kappa}\tilde{\xi}(t)\partial_t{{\lambda_{\eta}}}(t)\, dx.
\end{align*}
The estimate now follows using $\partial_t {{\lambda_{\eta}}} = -\partial_t \eta {{\lambda_{\eta}}}$ and by the uniform bounds of $\lambda_\eta$ and $\sigma$.

\end{proof}

\begin{proposition}[Solenoidal extension]
\label{TestFunction}
Let \eqref{eq:bounds} be satisfied and $\eta\in L^\infty(0,T;W^{1,2}(\omega))$. Then there exists a linear solenoidal extension operator 
\[
\test:\{\xi\in L^1(0,T;W^{1,1}(\omega)):\cor(\xi)\equiv0\}\to L^{1}(0,T;W^{1,1}( Q^\frac{\kappa}2)),
\]
 such that $\diverg{\test}(\xi-\cor(\xi))=0$ for all $\xi\in L^1(0,T;W^{1,1}(\omega))$ and 
$(\test(\xi-\cor(\xi)),\xi-\cor(\xi)
)\in \mathcal{V}_T$ for $\xi \in \Vcal_{K}$. 

Moreover, $\test(\xi-\cor(\xi))(t,x)=0$ for $(t,x)\in (0,T)\times Q_\frac{\kappa}2$ and it 
satisfies the following estimates for $q\in [1,\infty]$, $p\in (1,\infty)$ and $l\in \Nbb$.
\begin{align}
\|\test(\xi-\cor(\xi))\|_{L^{q}(0,T;L^p( Q^\frac{\kappa}2))}&\leq C \|\xi\|_{L^{q}(0,T;L^p(\omega))},
\\
\|\test(\xi-\cor(\xi))\|_{L^{q}(0,T;W^{1,p}( Q^\frac{\kappa}2))}&\leq C\Big( \|\xi\|_{L^{q}(0,T;W^{1,p}(\omega)}+\|\xi \nabla \eta \|_{L^{q}(0,T;L^{p}(\omega)}\Big),
\label{TestNablaBound}
\\
\|\partial_t\test(\xi-\cor(\xi))\|_{L^{q}(0,T;L^p(Q^\frac{\kappa}2))}
&\leq C
\Big(\|\partial_t \xi\|_{L^{q}(0,T;L^p(\omega)}
+\| \xi \partial_t\eta\|_{L^{q}(0,T;L^p(\omega))}\Big),
\label{TestTimeDBound}
\\
\nonumber
\|\nabla^2\test(\xi-\cor(\xi))\|_{L^{q}(0,T;L^p(Q^\frac{\kappa}2))}
&\leq C\Big(\|\nabla^2 \xi\|_{L^{q}(0,T;L^p(\omega)}+\| \xi \nabla^2 \eta\|_{L^{q}(0,T;L^p(\omega))}\Big)
\\
\nonumber
&\quad +C\Big(\| \abs{\nabla \xi}\abs{\nabla \eta})\|_{L^{q}(0,T;L^p(\omega))}
+\| \xi \abs{\nabla \eta}^2\|_{L^{q}(0,T;L^p(\omega))}\Big),
\\
\nonumber
\|\partial_t\nabla\test(\xi-\cor(\xi))\|_{L^{q}(0,T;L^p(Q^\frac{\kappa}2))}
&\leq C\Big(\|\partial_t\nabla\xi\|_{L^{q}(0,T;L^p(\omega)}+\| \xi \partial_t \nabla \eta\|_{L^{q}(0,T;L^p(\omega))}
\Big)
\\
\nonumber
&\quad +C\| \abs{\partial_t\xi}\abs{ \nabla \eta}+\abs{\nabla \xi}\abs{\partial_t \eta} +\abs{\xi\partial_t\eta} \abs{\nabla \eta}\|_{L^{q}(0,T;L^p(\omega))},
\\
\label{eq:normal}
\|\partial^l_\bn\test(\xi-\cor(\xi))\|_{L^{q}(0,T;W^{1,p}(Q^\frac{\kappa}2))}
&\leq C\|\test(\xi)\|_{L^{q}(0,T;W^{1,p}(Q^\frac{\kappa}2))},
\end{align}
whenever the right hand side is finite. Here $C$ depends only on $\alpha_\eta, \beta_\eta,$ and $\kappa$.
\end{proposition}
\begin{proof}

{\bf Construction:}

\noindent
The construction relies exclusively on the reference geometry, namely on $S_\kappa$ defined in Definition~\ref{def:coor}. Hence to keep the notation compact we will omit the dependence on the time variable $t$.
Moreover, without loss of generality we assume that
$\cor(\xi)=0$, since otherwise we replace $\xi$ by $\xi-\cor(\xi)$, for which we have 
\[
\cor(\xi-\cor(\xi))
=\frac{\int_{\Acal_\kappa}(\tilde{\xi}-{\cor}(\xi)){{\lambda_{\eta}}}\, dx}{\int_{\Acal_\kappa}{{\lambda_{\eta}}}\, dx}
=\frac{\int_{\Acal_\kappa}\tilde{\xi}{{\lambda_{\eta}}}\, dx}{\int_{\Acal_\kappa}{{\lambda_{\eta}}}\, dx}-{\cor}(\xi)
=
0.
\]
Hence, once the estimates are valid for $\xi$, such that $\cor(\xi)=0$ the estimates follow by Corollary~\ref{cor:cor} also for the case $\cor(\xi)\neq 0$.

First observe, that for the coordinates $s(x),\bfp(x)$ introduced in Definition~\ref{def:coor} we find
\[
\nabla s(x)=\partial_{\bn} s(x)\bn\text{ and }\nabla \bfp(x)=(\partial_{\bftau_i(\bfp(x))}\bfp(x)
)_{i=1,...,d-1}
\]
and (independent of $s(x)$)
\[
\divg(\bn(\bfp))=\sum_{i=1}^{d-1}\partial_{\bftau_i(\bfp)}\bn(\bfp)\cdot \bftau_i(\bfp).
\]
 For $y\in \omega$ and $x\in S_\kappa$ we find by the assumption on $\Omega$, that $y=y(x)$ if and only if $\bfp(x)=\bvphi(y)$ and so (wherever well defined)
\[
\partial_{\bn(\bfp)}\by(x)=0\text{ and so }\partial_{\bn(\bfp)} \xi(\by(x))\equiv 0.
\]
Next we introduce the operator:
\[
\overline{\test}(\xi)(x):=e^{(s(x)-\eta(\by(x)))\divg(\bn(\bfp(x)))}\tilde{\xi}(\bfp(x))\sigma_\kappa(s(x))\bn(\bfp(x)).
\]
Observe, that for $x\in \Omega_\eta\cap S_\kappa$, we find
\[
\overline{\test}(\xi)(x)=e^{(s(x)-\eta(\by(x)))\divg(\bn(\bfp(x)))}\tilde{\xi}(\bfp(x))\bn(\bfp(x)).
\]
In particular, for $x\in \partial\Omega_\eta$, we find $s(x)=\eta(\by(x))$ and hence
\[
\overline{\test}(\xi)(x)=\bn(\bfp(x))\xi(\by(x)),\;{x\in\partial\Omega_{\eta}}.
\] 
Using that {$\partial_{\bn(\bfp(x))} f(x)= -\partial_s f(\bfp,s)$} we find for $x\in Q^\frac{\kappa}{2}\cap S_{\kappa}$
\begin{align*}
\divg(\overline{\test}(\xi)(x))&=\nabla \Big((e^{(s(x)-\eta(\by(x)))\divg(\bn(\bfp(x)))}\tilde{\xi}(\bfp(x))\Big)\cdot\bn(\bfp(x))
\\
&\quad  +e^{(s(x)-\eta(t,\by(x)))\divg(\bn(\bfp(x)))}\tilde{\xi}(\bfp(x))\divg(\bn(\bfp(x)))
\\
&=-\partial_{s}\Big((e^{(s(x)-\eta(\by(x)))\divg(\bn(\bfp(x)))}\tilde{\xi}(\bfp(x))\Big)
\\
&\quad  +e^{(s(x)-\eta(\by(x)))\divg(\bn(\bfp(x)))}\tilde{\xi}(\bfp(x))\divg(\bn(\bfp(x)))
\\
&= 0.
\end{align*}
On $ \Acal_\kappa$ we find (by the same calculations) that
\[
\divg(\overline{\test}(\xi)(x)) = -e^{(s(x)-\eta(t,\by(x)))\divg(\bn(\bfp(x)))}\tilde{\xi}(\bfp(x))\sigma'_\kappa(s(x)),
\]
which has compact support in $ \Acal_\kappa $. Moreover, 
\[
\int_{ \Acal_\kappa} \divg(\overline{\test}(\xi)(x))\, dx
= -\int_{ \Acal_\kappa} {{\lambda_{\eta}}}(x) \tilde{\xi}(\bfp(x))\, dx 
=0.
\] 
Since $ \Acal_\kappa$ is by assumption a $C^2$ domain we can apply the \Bogovskij{} operator on this domain which we denote by $\bog_\kappa$. We define
\[
\test(\xi)(x):=\overline{\test}(\xi)(x)-\bog_\kappa(\divg(\overline{\test}(\xi)))(x).
\]

{\bf Estimates:}

\noindent
The estimates are quite standard relaying on the regularity of $\bvphi$, namely on the $C^2$-regularity of $\partial \Omega$. We give some details on the estimate in order show a clear dependence on $\eta$.

We start with the estimates of the time derivative of $\overline{\test}(\xi)$. We calculate
\[
\partial_t \test(\xi)=\partial_t\overline{\test}(\xi)-\bog_\kappa(\partial_t\divg(\overline{\test}(\xi))).
\]
The \Bogovskij{} operator is well defined due to the fact, that (formally)
\[
\int_{\Omega_{\eta(t)}}\partial_t\divg(\overline{\test}(\xi))\, dx= \int_{\Acal_\kappa}\partial_t\divg(\overline{\test}(\xi))\, dx =\partial_t\bigg(\int_{\Acal_\kappa}\divg(\overline{\test}(\xi))\, dx\bigg)=0. 
\]
We calculate further 
\begin{align*}
\partial_t\overline{\test}(\xi)(t, x) &= e^{(s(x)-\eta(t,\by(x)))\divg(\bn(\bfp(x)))}\partial_t\xi(t,\by(x))\bn(\bfp(x))
\\
&\quad  -\divg(\bn(\bfp(x)))\partial_t\eta(t,\by(x))  e^{(s(x)-\eta(t,\by(x)))\divg(\bn(\bfp(x)))}\xi(t,\by(x))\bn(\bfp(x)),
\end{align*}
which implies the pointwise estimates for $\partial_t\overline{\test}(\xi)$:
\begin{align}
\label{eq:pointwise}
\abs{\partial_t\overline{\test}(\xi)(t, x)}\leq c(\abs{\partial_t\xi(t,\bfp(x))}+\abs{\partial_t\eta(t,\bfp(x))}\abs{\xi(t,\bfp(x))}),
\end{align}
where the constant only depends on $\kappa,\alpha_\eta,\beta_\eta$ and the $C^2$-regularity of $\partial\Omega$.
 For the sake of better understanding we demonstrate that the assumption $\cor(\xi)=0$ is indeed without loss of generality. We estimate
\begin{align}
\label{eq:pointwise2}
\begin{aligned}
&\abs{\partial_t\overline{\test}(\xi-\cor(\xi))(t, x)}
\\
&\leq c\big(\abs{\partial_t\xi(t,\bfp(x))}+\abs{\partial_t\eta(t,\bfp(x))}(\abs{\xi(t,\bfp(x))}+\norm{\xi(t)}_{L^1})+\norm{\partial_t\eta(t) \xi(t)}_{L^1}\big).
\end{aligned}
\end{align}
 In order to estimate the \Bogovskij{} part we 
find by Theorem~\ref{BogovskiiTeorem1} (with a constant just depending on the Lipschitz constant of $\Acal_\kappa$) that
\begin{align*}
\norm{\bog_\kappa(\partial_t\divg(\overline{\test}(\xi-\cor(\xi))))}_{L^p(\Omega_\eta)}
&=\norm{\bog_\kappa(\partial_t\divg(\overline{\test}(\xi-\cor(\xi))))}_{L^p(\Acal_\kappa)}
\\
&\leq c\norm{\divg(\partial_t\overline{\test}(\xi-\cor(\xi))))}_{\hat{W}^{-1,p}(\Acal_\kappa)}
\\
&=c\norm{\partial_t\overline{\test}(\xi-\cor(\xi))}_{L^p(\Acal_\kappa)}.
\end{align*}
and so the estimate on $\partial_t{\test}(\xi-\cor(\xi))$ follows by \eqref{eq:pointwise2}.

The estimates on $\nabla{\test}(\xi-\cor{\xi}),\nabla^2{\test}(\xi-\cor{\xi})$ and $\partial_t\nabla{\test}(\xi-\cor{\xi})$ are analogous and we skip the details here. Observe that due to the compact support of $\divg(\overline{\test}(\xi-\cor(\xi))))$, by Gauss theorem  we find that
\[
\int_{\Omega_{\eta(t)}}\nabla^l\divg(\overline{\test}(\xi-\cor{\xi}))\, dx=0=\int_{\Omega_{\eta(t)}}\partial_t\nabla^l\divg(\overline{\test}(\xi-\cor{\xi}))\, dx; 
\]
hence $\bog_\kappa$ is always well defined.

Clearly the normal derivatives of the constructed function $\overline{\test}(\xi-\cor(\xi)))$ depend on the estimates of the derivatives of $\sigma_\kappa$ and not on the regularity of the derivatives of $\eta$. Since the \Bogovskij{} theorem transfers the regularity to $\test(\xi-\cor(\xi))$ with no further loss, \eqref{eq:normal} follows with according dependences on the higher order derivatives of $\sigma_\kappa$.

\end{proof}
We include the following corollary that will be necessary for our compactness result (See Section~\ref{sec:auba}).
\begin{corollary}[Smooth Solenoidal Extension]
\label{cor:extend}
Let $a,\;r\in[2,\infty]$, $p,\;q\in (1,\infty)$ and $s\in [0,1]$. Assume that $\eta\in L^r([0,T];W^{2,a}(\omega))\cap W^{1,r}([0,T];L^a(\omega))$, such that $\alpha(\Omega)+\kappa\leq\alpha_\eta\leq  \eta\leq \beta_\eta\leq \beta(\Omega)-\kappa$. 

Let $b\in W^{s,p}(\omega))$ and take $(b)_\delta$ as a smooth approximation of $b$ in $\omega$. Then $E_{\eta,\delta}(b):=\test((b)_\delta-\cor((b)_\delta))$ satisfies all the regularity of Proposition~\ref{TestFunction}. In particular
\[
\norm{E_{\eta(t),\delta}(b)-\test(b-\cor(b))}_{L^p(Q^\frac{\kappa}{2})}\leq c\norm{(b)_\delta-b}_{L^p(\omega)}
\]
and
\[
\norm{\partial_t E_{\eta(t),\delta}(b)}_{L^r(0,T;L^a(Q^\frac{\kappa}{2})}\leq c\norm{(b)_\delta \partial_t\eta(t)}_{L^r(0,T;L^a(\omega)}
\]
uniformly in $t\in (0,T)$.
\end{corollary}

We include the following technical lemma, that will be necessary for the regularity result.
\begin{lemma}
\label{lem:IbP}
Let $p,\;\tilde{a}\in (1,\infty)$ such that $p'< \tilde{a}\leq \frac{dp'}{d-p'}$ if $p'<d$, and $p'< \tilde{a}<\infty$ otherwise, and let the assumptions of Proposition~\ref{TestFunction} be satisfied. Assume additionally that $\eta\in C^{0,\theta}(\omega)\cap W^{1,\frac{\tilde{a}p}{\tilde{a}p-\tilde{a}-p}}(\omega)$ and $\bu\in W^{1,p'}(\Omega_\eta)$ then the above constructed test function satisfies 
\begin{align}
\label{eq:partial1}
&\absBB{\int_{\Omega_\eta} \bu\cdot \test (D^s_{h,e}\xi-\cor(D^s_{h,e}\xi))\, dx}
\leq c(h^{\theta-s}+ \norm{D^s_{h,e}\eta}_{W^{1,\frac{\tilde{a}p}{\tilde{a}p-\tilde{a}-p}}(\omega)})\norm{\bu}_{W^{1,p'}(\Omega_{\eta})}\norm{\xi}_{L^p(\omega)} 
\intertext{and in case $\partial_t \xi\in L^p(\omega)$}
\label{eq:partial2}
&\absBB{\int_{\Omega_\eta} \bu\cdot \partial_t\test (D^s_{h,e}\xi-\cor(D^s_{h,e}\xi))\, dx}
\\
&\quad \leq c\Big(\big(h^{\theta-s}+\norm{D^s_{h,e}\eta}_{W^{1,\frac{\tilde{a}p}{\tilde{a}p-\tilde{a}-p}}(\omega)}\big)\norm{\partial_t\xi}_{L^p(\omega)}\norm{\bu}_{W^{1,p'}(\Omega_{\eta})}
 +\norm{D^s_{h,e}\xi\partial_t\eta}_{L^{\tilde{a}'}(\omega)}\norm{\bu}_{W^{1,p'}(\Omega_{\eta})}\Big).
\end{align}
 The constants are only depending on $\alpha_\eta,\beta_\eta$, $\kappa$ and (linearly) on $\norm{\eta}_{C^{1,\theta}(\omega)}$.
\end{lemma}
\begin{proof}
In the following we use the abbreviation $\delta_hf(y)=(f(y+e_i h)-f(y))$ for $i=1,2$. Moreover, since all estimates are done pointwise in time, we omit the dependence on $t$ of $\eta$ and $\Omega_{\eta}$.
First, since the support of $\test (\delta_h\xi-\cor(\delta_h\xi))$ is $S_{\frac{\kappa}2}$, we can use the coordinates $(\bfp,s)$ on the full support of $\test (\delta_h\xi-\cor(\delta_h\xi))$. We will use the following change of coordinates $\psi_\eta\circ\bfPhi:\omega\times (\alpha+\kappa/2,0]\to \Omega_\eta$ in order to be able to do integration by parts. Hence
\begin{align*}
&\int_{\Omega_\eta} \bu\cdot \test (\delta_h\xi-\cor(\delta_h\xi))\, dx
\\
&\quad 
=\int_{\alpha+\kappa/2}^0\int_{\omega} (\bu\cdot \test (\delta_h\xi-\cor(\delta_h\xi)))\circ\psi_\eta\circ\bfPhi\abs{\det(\nabla(\psi_\eta\circ\bfPhi)}\, dy\, ds
\end{align*}
We will use the following abbreviations for the sake of a better overview:
\[
\alpha=\alpha(\Omega), \quad \tilde{\gamma}(s,y):=\abs{\det(\nabla(\psi_\eta\circ\bfPhi)(s,y)}\text{ and }\tilde{\lambda}_{\eta}(s,y):=e^{(s-\eta(\by)\divg_x(\bn(y))}\sigma_\kappa'(s)\tilde{\gamma}(s,y).
\] 
Hence, we calculate
\begin{align*}
\norm{{{\lambda_{\eta}}}}_{L^{1}(\Acal_\kappa)} \cor(\delta_h\xi)&=\int_{\alpha+\kappa/2}^{\alpha+\kappa}\int_{\omega} {{\lambda_{\eta}}} \circ\psi_\eta\circ\bfPhi \delta_h\xi \abs{\det(\nabla(\psi_\eta\circ\bfPhi)}\, dy\, ds
\\
&=:\int_{\alpha+\kappa/2}^{\alpha+\kappa}\int_{\omega} \delta_h\xi(y)\tilde{\lambda}_{\eta}(s,y)\, dy\, ds
=\int_{\alpha+\kappa/2}^{\alpha+\kappa}\int_{\omega} \xi(y)\delta_h\tilde{\lambda}_{\eta}(s,y)\, dy\, ds.
\end{align*}
where we used summation by parts formula for finite differences.
Therefore,
\[
\delta_h\xi-\cor(\delta_h\xi)=\delta_h(\xi-\cor(\xi))-\frac{\int_{\alpha+\kappa/2}^{\alpha+\kappa}\int_{\omega} \xi(y)\delta_h(\tilde{\lambda}_{\eta})\, dy}{\norm{{{\lambda_{\eta}}}}_{L^{1}(\Acal_\kappa)}}.
\]
And so
\begin{align*}
&\overline{\test}(\delta_h\xi-\cor(\delta_h\xi))\circ\psi_\eta\circ\bfPhi(s,y)
=
\delta_h(\overline{\test}(\xi-\cor(\xi))\circ\psi_\eta\circ\bfPhi(s,y))
\\
&\quad - \delta_h\Big(e^{(s-\eta(\by)\divg_x(\bn(y))}\Big)\sigma_\kappa(s)\Big(\xi-\mean{\xi}_{{{\lambda_{\eta}}}}\Big)\bn(y)
\\
&\quad -e^{(s-\eta(\by)\divg_x(\bn(y))}\sigma_\kappa(s)\Big(\xi-\mean{\xi}_{{{\lambda_{\eta}}}}\Big)\delta_h(\bn(y))
\\
&\quad -e^{(s-\eta(\by)\divg_x(\bn(y))}\sigma_\kappa(s)\frac{\int_{\alpha+\kappa/2}^{\alpha+\kappa}\int_{\omega} \xi(y)\delta_h(\tilde{\lambda}_{\eta})\, dy\, ds}{
\norm{{{\lambda_{\eta}}}}_{L^{1}(\Acal_\kappa)}}\bn(y)
\\
&=:\delta_h(\overline{\test}(\xi)\circ\psi_\eta\circ\bfPhi(s,y))+\Tcal_{1}+\Tcal_{2}+\Tcal_{3}.
\end{align*}
The partial summation and the H\"older's inequality imply that
\begin{align*}
&\absBB{\int_{\Omega_\eta} \bu\cdot \test (\delta_h\xi-\cor(\delta_h\xi))\, dx}
\\
&\quad \leq \absBB{\int_{\alpha+\kappa/2}^0\int_{\omega} (\bu\circ\psi_\eta\circ\bfPhi\cdot \delta_h\overline{\test} (\xi-\cor(\xi))\circ\psi_\eta\circ\bfPhi\,\tilde{\gamma}\, dy\, ds}
\\
&\qquad +\absBB{\int_{\alpha+\kappa/2}^0\int_{\omega} (\bu\circ\psi_\eta\circ\bfPhi\cdot (\Tcal_1+\Tcal_2+\Tcal_3))\tilde{\gamma}\, dy\, ds}
\\
&\qquad +\absBB{\int_{\Acal_\kappa}\bu\cdot \bog_\kappa (\divg( \overline{\test}(\delta_h\xi-\cor(\delta_h\xi))\, dx}
\\
&\quad =(I)+(II)+(III).
\end{align*}
Recall, that $p'< \tilde{a}\leq \frac{dp'}{d-p'}$ (if $p<d$ and no upper bound otherwise). Observe, that 
\[
\abs{\delta_h\eta}\leq ch^\theta,\quad \abs{\gamma}+\abs{\nabla (\psi_\eta\circ\bfPhi)}\leq c(1+\abs{\nabla \eta})\text{ and }\abs{\delta_h \tilde{\gamma}}\leq c\abs{\delta_h \nabla \eta}
\]
and
\begin{align*}
\abs{\delta_h\bu\circ\psi_\eta\circ\bfPhi}&\leq \absBB{\frac{\bu(\psi_\eta\circ\bfPhi(x+h))-\bu(\psi_\eta\circ\bfPhi(x))}{\psi_\eta\circ\bfPhi(x+h)-\psi_\eta\circ\bfPhi(x)}}\abs{\psi_\eta\circ\bfPhi(x+h)-\psi_\eta\circ\bfPhi(x)}
\\
&\leq ch^\theta\dashint_{\psi_\eta\circ\bfPhi(x)}^{\psi_\eta\circ\bfPhi(x+h)}\abs{\nabla u}\, ds.
\end{align*}
We estimate $(I)$ using partial integration, the above inequality, H\"older's inequality for $\frac{1}{p}+\frac{1}{\tilde{a}}+\frac{\tilde{a} p-p-\tilde{a}}{\tilde{a} p}=1$ and Sobolev embedding:
\begin{align}
\label{eq:I}
\begin{aligned}
(I)
&=\absBB{\int_{\alpha+\kappa/2}^0\int_{\omega} ( \delta_h(\bu\circ\psi_\eta\circ\bfPhi)\cdot \overline{\test} (\xi-\cor(\xi))\circ\psi_\eta\circ\bfPhi)\, \tilde{\gamma}
\\
&\quad +
(\bu\circ\psi_\eta\circ\bfPhi)\cdot \overline{\test} (\xi-\cor(\xi))\circ\psi_\eta\circ\bfPhi)\,\delta_h \tilde{\gamma} \, dy\, ds}
\\
&\leq ch^\theta \norm{\bu}_{W^{1,p'}(\Omega_{\eta})}\norm{\xi}_{L^p(\omega)} 
+ c\norm{\bu}_{L^{\tilde{a}}(\Omega_{\eta})}\norm{\xi}_{L^p(\omega)}
\norm{\delta_h\eta}_{W^{1,\frac{\tilde{a}p}{\tilde{a}p-\tilde{a}-p}}(\omega)}
\\
&\leq c\big(h^\theta+ \norm{\delta_h\eta}_{W^{1,\frac{\tilde{a}p}{\tilde{a}p-\tilde{a}-p}}(\omega)}\big)\norm{\bu}_{W^{1,p'}(\Omega_{\eta})}\norm{\xi}_{L^p(\omega)}.
\end{aligned}
\end{align}
We further estimate $(II)$ by using, in a rather straightforward manner, the fact that $\abs{\delta_h g}\leq h^\theta \norm{ g}_C^{0,\theta}$ 
\begin{align*}
\abs{\Tcal_1\tilde{\gamma}}&\leq c(h+\abs{\delta_h\eta})(\abs{\xi}\tilde{\gamma}+\abs{\mean{\xi}_{\tilde{\lambda}_{\eta}}\tilde{\lambda}_{\eta}})\leq  ch^\theta(\abs{\xi}\tilde{\gamma}+\abs{\mean{\xi}_{\tilde{\lambda}_{\eta}}\tilde{\lambda}_{\eta}}),
\\
\abs{\Tcal_2\tilde{\gamma}}&\leq ch(\abs{\xi}\tilde{\gamma}+\abs{\mean{\xi}_{\tilde{\lambda}_{\eta}}\tilde{\lambda}_{\eta}}), 
\\
\abs{\Tcal_3\tilde{\gamma}}&\leq c\frac{\abs{\tilde{\lambda}_{\eta}}}{\norm{{{\lambda_{\eta}}}}_{L^{1}(\Acal_\kappa)}}\norm{\xi}_{L^p(\omega)}\norm{\delta_h\eta}_{W^{1,p'}(\Omega)}.
\end{align*}
This implies 
\begin{align*}
\sum_{i=1}^2\int_{\alpha+\kappa/2}^{0}\int_{\omega}\abs{\Tcal_i\tilde{\gamma}}^p\, dy, \,ds &\leq ch^{\theta p} \norm{\xi}_{L^p(\omega)}^p\bigg(1+\frac{\norm{\tilde{\lambda}_{\eta}}_{L^{p'}(\Acal_\kappa)}^p}{\norm{\tilde{\lambda}_{\eta}}_{L^{1}(\Acal_\kappa)}^p}\bigg),
\\
\int_{\alpha+\kappa/2}^{0}\int_{\omega}\abs{\Tcal_3\tilde{\gamma}}^p\, dy, \,ds &\leq c\frac{\norm{\tilde{\lambda}_{\eta}}_{L^{p}([\alpha+\kappa/2,{0}]\times\omega)}^p}{\norm{{{\lambda_{\eta}}}}_{L^{1}(\Acal_\kappa)}^p}\norm{\xi}_{L^p(\omega)}^p\norm{\delta_h\eta}_{W^{1,p'}(\omega)}^p.
\end{align*}
Hence, we find
by H\"older's and Poincar\'e's inequality that
\begin{align}
\label{eq:lpbound}
(II)\leq c \norm{\bu}_{L^{p'}(\Omega_{\eta})}\sum_{i=1}^3\norm{\Tcal_i\tilde{\gamma}}_{L^p((\alpha+\kappa/2,0)\times \omega)} &\leq c(h^\theta+\norm{\delta_h\eta}_{W^{1,p'}(\omega)})\norm{\xi}_{L^p(\omega)} \norm{\bu}_{W^{1,p'}(\Omega_{\eta})}.
\end{align}
The estimates on $(I)$ and $(II)$ allow to estimate the \Bogovskij{} term $(III)$.
This is possible since due to Theorem~\ref{BogovskiiTeorem1} and due to the compact support of $\sigma'$ in $\Acal_\kappa$ we find
\begin{align*}
(III):&=\abs{\skp{\bu,\bog_\kappa(\divg( \overline{\test}(\delta_h\xi-\cor(\delta_h\xi)))}}
\\
&\leq \norm{\bu}_{W^{1,p'}(\Acal_\kappa)}\norm{\bog_\kappa(\divg( \overline{\test}(\delta_h\xi-\cor(\delta_h\xi)))}_{\mathring{W}^{-1,p}(\Acal_\kappa)}
\\
&\leq c\norm{\bu}_{W^{1,p'}(\Acal_\kappa)}\norm{\divg(\overline{\test}(\delta_h\xi-\cor(\delta_h\xi)))}_{\mathring{W}^{-2,p}(\Acal_\kappa)}
\\
&\leq c\norm{\bu}_{W^{1,p'}(\Acal_\kappa)}\norm{\overline{\test}(\delta_h\xi-\cor(\delta_h\xi))}_{\mathring{W}^{-1,p}(\Acal_\kappa)}.
\end{align*}
Now take $q\in W^{1,p'}_0(\Acal_\kappa)$, with $\norm{q}_{W^{1,p'}(\Acal_\kappa)}\leq 1$ arbitrary.
From the calculations above, i.e.\ by replacing $\bu$ by $q$ in~\eqref{eq:I} and~\eqref{eq:lpbound} we find
\begin{align*}
&\skp{\overline{\test}(\delta_h\xi-\cor(\delta_h\xi)),q}
\\
&\quad =\int_{\alpha+\kappa/2}^{0}\int_{\omega} (\delta_h(\overline{\test}(\xi-\cor(\xi)))\cdot q)\circ\psi_\eta\circ\bfPhi \tilde{\gamma}\, dy\, ds + \sum_{i=1}^3\int_{\alpha+\kappa/2}^{\alpha+\kappa}\int_{\omega} \Tcal_i \cdot q\circ\psi_\eta\circ\bfPhi \tilde{\gamma}\, dy\, ds
\\
&\quad =\int_{\alpha+\kappa/2}^{0}\int_{\omega} (
\overline{\test}(\xi-\cor(\delta_h\xi))\circ\psi_\eta\circ\bfPhi\cdot \delta_h(q \circ\psi_\eta\circ\bfPhi\tilde{\gamma})\, dy\, ds
\\
&\quad 
 + \sum_{i=1}^3\int_{\alpha+\kappa/2}^{\alpha+\kappa}\int_{\omega} \Tcal_i \cdot q\circ\psi_\eta\circ\bfPhi \tilde{\gamma}\, dy\, ds
 \leq c(h^\theta+ \norm{\delta_h\eta}_{W^{1,\frac{\tilde{a}p}{\tilde{a}p-\tilde{a}-p}}(\omega)})\norm{\xi}_{L^p(\omega)}. 
\end{align*}
But so
\begin{align*}
(III)&\leq c\norm{\bu}_{W^{1,p'}(\Acal_\kappa)}\norm{\overline{\test}(\delta_h\xi-\cor(\delta_h\xi))}_{\mathring{W}^{-1,p}(\Acal_\kappa)}
\\
&\leq c\norm{\bu}_{W^{1,p'}(\Acal_\kappa)}\Big(h^\theta+\norm{\delta_h\eta}_{W^{1,\frac{\tilde{a}p}{\tilde{a}p-\tilde{a}-p}}(\omega)}\Big)\norm{\xi}_{L^p(\omega)}.
\end{align*}
This finishes the proof of \eqref{eq:partial1}. For the time derivative we use the fact that
\begin{align}
\label{eq:cortime}
\begin{aligned}
&\partial_t(\delta_h\xi-\cor(\delta_h\xi))=(\delta_h\partial_t\xi-\cor(\delta_h\partial_t\xi))
-\frac{\mean{\delta_h\tilde{\xi}(t)}_{{\lambda_{\eta}}}}{\norm{{{\lambda_{\eta}}}(t)}_{L^1}}\int_{ \Acal_\kappa}\partial_t{{\lambda_{\eta}}}(t)\, dx
\\
&\qquad  + \frac{1}{\norm{{{\lambda_{\eta}}}(t)}_{L^1}}\int_{ \Acal_\kappa}\delta_h\tilde{\xi}(t)\partial_t{{\lambda_{\eta}}}(t)\, dx =:\cor(\delta_h\partial_t\xi))+K(t),
\end{aligned}
\end{align}
and hence
\begin{align*}
&\partial_t\overline{\test}(\delta_h\xi-\cor(\delta_h\xi))(t, x)
\\
&\quad = \sigma(s(x))e^{(s(x)-\eta(t,\by(x)))\divg(\bn(\bfp(x)))}(\delta_h(\partial_t\xi)-\cor(\delta_h(\partial_t\xi))(t,\by(x))\bn(\bfp(x))
\\
&\quad  -\sigma(s(x))\divg(\bn(\bfp(x)))\partial_t\eta(t,\by(x))  e^{(s(x)-\eta(t,\by(x)))\divg(\bn(\bfp(x)))}(\delta_h\xi-\cor(\delta_h\xi))(t,\by(x))\bn(\bfp(x))
\\
&\quad  -\sigma(s(x))\divg(\bn(\bfp(x)))e^{(s(x)-\eta(t,\by(x)))\divg(\bn(\bfp(x)))}K(t)\bn(\bfp(x))
\\
&=(A) +(B)+(C).
\end{align*}
The estimates on $(A)$ follows by \eqref{eq:partial1}. We proceed with the straightforward estimates
\[
 \abs{(B)}\leq c\abs{\partial_t\eta(t,\by(x))}(\abs{\delta_h(\xi(t,\by(x))}+\norm{\delta_h(\xi(t))}_{L^1(\Acal_\kappa)}),
\]
and
\[
\abs{(C)}\leq c\norm{\delta_h\xi(t)\partial_t\eta(t)}_{L^{1}(\omega)}.
\]
Hence, we find by \eqref{eq:partial1}, the estimates on $(B)$ and $(C)$, H\"older's inequality and Sobolev embedding that
\begin{align*}
&\absBB{\int_{\Omega_\eta} \bu\cdot \partial_t\overline{\test} (\delta_h\xi-\cor(\delta_h\xi))\, dx}
\\
&\quad \leq \Big(h^\theta+\norm{\delta_h\eta}_{W^{1,\frac{\tilde{a}p}{\tilde{a}p-\tilde{a}-p}}(\omega)}\Big)\norm{\partial_t\xi}_{L^p(\omega)}\norm{\bu}_{W^{1,p'}(\Omega_{\eta})} 
 +c\norm{\bu}_{L^{\tilde{a}}(\Omega_{\eta})}\norm{\delta_h\xi\partial_t\eta}_{L^{\tilde{a}'}(\omega)}
\\
&\quad \leq c\Big(\Big(h^\theta+\norm{\delta_h\eta}_{W^{1,\frac{\tilde{a}p}{\tilde{a}p-\tilde{a}-p}}(\omega)}\Big)\norm{\partial_t\xi}_{L^p(\omega)}+\norm{\delta_h\xi\partial_t\eta}_{L^{\tilde{a}'}(\omega)}\Big)\norm{\bu}_{W^{1,p'}(\Omega_{\eta})} .
\end{align*}
The \Bogovskij{} part will be estimated once more in form of negative norms using that
\begin{align*}
&\sup_{\norm{q}_{W^{1,p'}(\Acal_\kappa)}\leq 1}\skp{\partial_t\overline{\test}(\delta_h\xi-\cor(\delta_h\xi)))),q}
\\
 &\quad \leq c\Big(\big(h^\theta+\norm{\delta_h\eta}_{W^{1,\frac{\tilde{a}p}{\tilde{a}p-\tilde{a}-p}}(\omega)}\big)\norm{\partial_t\xi}_{L^p(\omega)}+\norm{\delta_h\xi\partial_t\eta}_{L^{\tilde{a}'}(\omega)}\Big),
\end{align*} 
which finishes the proof.
\end{proof}

\section{The regularity result}
\label{sec:thm2}
\subsection{Estimates for the structure}
In this section we explore the consequences of the energy inequality \eqref{EI}.
\begin{lemma}[Uniform Korn's inequality]
For every $\bu\in \mathcal{V}_F$ such that $\bu(t,\bvphi_{\eta}(t,.))=\xi\bn$ the following Korn's equality holds:
\begin{equation}\label{Korn}
\|\nabla\bu\|^2_{L^2{\Omega_{\eta}(t)}}=2\|{\rm sym}\nabla\bu\|^2_{L^2{\Omega_{\eta}(t)}}.
\end{equation}
\end{lemma}
\begin{proof}
We follow the idea from \cite[Lemma~6]{CDEM} and compute:
$$
\int_{\Omega_{\eta(t)}}|{\rm sym}\nabla\bu|^2\, dx
=\frac{1}{2}\Big (\int_{\Omega_{\eta(t)}}|\nabla u|^2\, dx
+\int_{\Omega_{\eta}(t)}\nabla^{T}\bu:\nabla\bu\, dx \Big ).
$$
Therefore it remains to show that the second term is zero:
\begin{align*}
&\int_{\Omega_{\eta}(t)}\nabla^{T}\bu:\nabla\bu\, dx
=\sum_{i,j=1}^2\int_{\Omega_{\eta}(t)}\partial_j u_i\partial_i u_j\, dx
\\
&\quad =-\sum_{i,j=1}^2\int_{\Omega_{\eta}(t)}\partial_j\partial_i u_i u_j\, dx+\int_{\partial\Omega_{\eta}(t)}\partial_ju_in_iu_jdS.
=\int_{\partial\Omega_{\eta}(t)}(\nabla\bu)\bn\cdot\bu dS
\end{align*}
Now using the no-slip condition \eqref{kinematic} and the incompressibility condition we deduce 
\\
$\int_{\partial\Omega_{\eta}(t)}(\nabla\bu)\bn\cdot\bu dS=0$ (see \cite[ Lemma A.5]{LenRuz}) and therefore the Korn's equality holds.
\end{proof}
In the following we exploit the energy estimate \eqref{EI-weak}. In particular, the number $C_0$,  which depends only on the initial conditions, always refers to this energy bound.
\begin{lemma}\label{W14Bound}
Let $(\bu,\eta)$ be such that energy inequality \eqref{EI-weak} is satisfied.
\\
 Then $\eta\in L^{\infty}(0,T;W^{1,4}(\omega))$ and $\|\eta\|_{L^{\infty}_tW^{1,4}_x}\leq cC_0$, where $c$ depends only on $\bvphi$.
\end{lemma}
\begin{proof}
The boundedness of $\|\eta\|_{L^{\infty}_tL^2_x}$ follows directly from the energy inequality~\eqref{EI-weak}. Now, we use \cite[Theorem 3.3-2.]{CiarletBook3} to conclude that by the definition of $\Acal$ and \eqref{EI-weak}:
$$
\int_{\omega}|\bG(\eta(t,.))|^2\, dy \leq c\int_{\omega}\mathcal{A}\bG(\eta(t,.)):\bG(\eta(t,.))\, dy\leq cC_0,
$$ 
here the constant $c$ just depends on the Lam\'e constants and the geometry of $\partial \Omega$.
If $\partial_{\alpha}\bn\neq 0$ we my use the bound for $G_{\alpha\alpha}(\eta)$ and \eqref{ChangeMetric} to get the bounds for $\|\partial_{\alpha}\eta(t)\|_{L^4(\omega)}$ and $\|\eta(t)\|_{L^4(\omega)}$ uniform in $t$. Using these bounds, again \eqref{ChangeMetric} and the bound for $G_{\beta\beta}(\eta)$ above for $\beta\neq\alpha$ we finish the proof.

If $\partial_1\bn=\partial_2\bn=0$, we get the bound for $\|\nabla\eta\|_{L^4}$ directly from \eqref{ChangeMetric} and the boundedness of $\int_{\omega}|\bG(\eta(t,.))|^2$. However, since $\|\eta\|_{L^{\infty}_tL^2_x}$ is also bounded (using the bounds on $\partial_t\eta$ in \eqref{EI}), the Lemma follows also by the Poincar\' e inequality.  
\end{proof}

\begin{lemma}\label{H2Bound}
Let $(\bu,\eta)$ be such that energy inequality \eqref{EI-weak} is satisfied.
Then if $\gamma(\eta)\neq 0$ we have $\eta(t)\in H^2(\omega)$. Moreover,
$$
\sup_{t\in [0,T]}\int_\omega \gamma^2(\eta)\abs{\nabla^2\eta}^2\, dy\leq cC_0.
$$ 
 where $c$ depends only on $\bvphi$.
\end{lemma}
\begin{proof}
We can again use Theorem 3.3-2. from \cite{CiarletBook3} and work with bounds on $\mathbf{R}$.
From \eqref{DeformedTangent} we compute:
\begin{equation}\label{TangentDerivatives}
\partial_{\beta}\ba_{\alpha}(\eta)
=\partial^2_{\alpha\beta}\bvphi+\partial^2_{\alpha\beta}\eta\bn
+\partial_{\alpha}\eta\partial_{\beta}\bn+\partial_{\beta}\eta\partial_{\alpha}\bn
+\eta\partial^2_{\alpha\beta}\bn,\; \alpha,\beta=1,2.
\end{equation}
Using \eqref{DeformedTangent}, \eqref{NormalDeformed}, \eqref{CurvatureChange} and the definition of $\gamma$ from Definition~\ref{def:coor}
we have
\begin{equation}\label{Curvature2}
\begin{array}{c}
\displaystyle{R_{\alpha\beta}(\eta)=\frac{1}{|\ba_1\times\ba_2|}\partial^2_{\alpha\beta}\eta\Big (|a_1\times a_2|+
\eta(\bn\cdot(\ba_1\times\partial_2\bn+\partial_1\bn\times\ba_2))
+\eta^2\bn\cdot(\partial_1\bn\times\partial_2\bn)\Big )}
\\
\displaystyle{+P_0(\eta,\nabla\eta)=:\gamma(\eta)\partial^2_{\alpha\beta}\eta+P_0(\eta,\nabla\eta)},
\end{array}
\end{equation}
where $P_0$ is a polynomial of order three in $\eta$ and $\nabla\eta$ such that all terms
are at most quadratic in $\nabla\eta$, and the coefficients of $P_0$ depend on $\bvphi$.  

From Lemma \ref{W14Bound} we gain in particular by Sobolev embedding that $\|\eta\|_{L^{\infty}_tL^{\infty}_x}$ and $\|\nabla\eta\|_{L^{\infty}_tL^4_x}$ are bounded by the energy. Therefore
$$
\sup_{t\in [0,T]}\int_\omega \gamma^2(\eta)\abs{\nabla^2\eta}^2\, dy\leq c(\|{\bf R}\|_{L^{\infty}_tL^2_x}+\|P_0(\eta,\nabla\eta)\|_{L^{\infty}_tL^2_x})\leq cC_0.
$$
\end{proof}
\begin{remark}
\label{rem:T}
By definition we know that $\gamma(\eta)>0$, as long as 
\begin{equation}\label{CoercitivityCondition}
\eta(\bn\cdot(\ba_1\times\partial_2\bn+\partial_1\bn\times\ba_2))
+\eta^2\bn\cdot(\partial_1\bn\times\partial_2\bn)>-\frac{1}{|\ba_1\times\ba_2|}.
\end{equation}
Therefore, since $\gamma(\eta_0)>0$, it follows that there exists a $c_2$ (depending on $\bvphi$ only) such that if $\|\eta-\eta_0\|_{L^{\infty}_tL^{\infty}_x}\leq c_2$, then~\eqref{CoercitivityCondition} is satisfied and hence $\gamma(\eta)>0$. Finally, the energy estimate allows to deduce directly (combining the $L^\infty_t W_x^{1,4}$ and the $W^{1,\infty}_tL^2_x$ estimate), that in dependence of the initial configuration there is a minimal time interval $(0,T)$ for which $\|\eta-\eta_0\|_{L^{\infty}_tL^{\infty}_x}\leq c_2$ is always satisfied.
\end{remark}

Similarly as in previous Lemma, let us write form $a_b$ defined by $\eqref{ElasticForms}$ as a sum of the bilinear form in second derivatives plus the remainder. 
We calculate the Fr\' echet derivative of $\mathbf{R}$:
$$
R_{\alpha\beta}(\eta)\xi=\gamma(\eta)\partial^2_{\alpha\beta}\xi+\gamma(\xi)\partial^2_{\alpha\beta}\eta+P_0'(\eta,\nabla\eta)\xi.
$$
Therefore we have
\begin{equation}\label{BendingFrom}
\begin{array}{c}
\displaystyle{a_b(t,\eta,\xi)=\frac{h^3}{24}\int_{w}\Big [\mathcal{A}\big (\gamma(\eta)\nabla^2\eta):\big (\gamma(\eta)\nabla^2\xi\big )
+\mathcal{A}\big (\gamma(\eta)\nabla^2\eta):\big (\gamma(\xi)\nabla^2\eta\big )}
\\ \\
\displaystyle{+\Big (\mathcal{A}\big (\gamma(\eta)\nabla^2\eta):P_0'(\eta,\nabla\eta)\xi
+\mathcal{A}\big (P_0(\eta,\nabla\eta)):\big (\gamma(\xi)\nabla^2\eta\big )\Big )}
\\ \\
\displaystyle{+\mathcal{A}\big (P_0(\eta,\nabla\eta)):\big (\gamma(\eta)\nabla^2\xi\big )
+\mathcal{A}\big (P_0(\eta,\nabla\eta)):P_0'(\eta,\nabla\eta)\xi\Big ] dy}
\\ \\
=a_b^1(\eta;\nabla^2\eta,\nabla^2\xi)
+a_b^2(\eta,\nabla^2\eta;\xi)
+a^3_b(\eta,\nabla\eta,\nabla^2\eta;\xi,\nabla\xi)
\\ \\
+a^4_b(\eta,\nabla\eta;\nabla\xi,\nabla^2\xi)+a^5_b(\eta,\nabla\eta,\xi,\nabla\xi).
\end{array}
\end{equation}
We take $\xi=D^{s}_{-h}D^{s}_h\eta$, $0<s<1/2$, and obtain the following estimates.
\begin{lemma}\label{StructureRegEstimate}
Let $\eta\in H^2(\omega)$ such that $\gamma(\eta)\neq 0$. Then for every $h>0$, $0<s<1/2$ the following inequality holds:
$$
a_b(t,\eta,D^{-s}_hD^{s}_h\eta)\geq \|D^{s}_h\nabla^2\eta\|_{L^2(\omega)}-C(\|\eta\|_{H^2(\omega)}).
$$
\end{lemma}
\begin{proof}
Since all estimates in this lemma are uniform in $t$ for simplicity of notation, we omit the $t$ variable in this proof. First we use the fact that since $\omega\subset \Rbb^2$ Sobolev embedding implies $\|D^{s}_h\eta\|_{L^{\infty}}\leq c\|\eta\|_{H^2(\omega)}$ and $\|D^{s}_{-h}D^{s}_h\eta\|_{L^{\infty}(\omega)}\leq c\|\eta\|_{H^2(\omega)}$. Due to Sobolev embedding the estimate is uniform in $h$ for all $s\in (0,1/2)$. This and the integration by parts formula for the finite differences can be used to estimate $a_b^1$: 
$$
a^1_b(\eta;\nabla^2\eta,\nabla^2D^{s}_{-h}D^{s}_h\eta)\geq C\int_{\omega}|D^{s}_h\nabla^2\eta|^2\, dy
-C\|D^{s}_h\gamma(\eta)^2\|_{L^{\infty}}\|\nabla^2\eta\|_{L^2}\|D^{s}_h\nabla^2\eta\|_{L^2}
$$
$$\geq \frac{C}{2}\|D^{s}_h\nabla^2\eta\|^2_{L^2}-C\|D^{s}_h\gamma(\eta)^2\|^2_{L^{\infty}}\|\nabla^2\eta\|^2_{L^2}
\geq \frac{C}{2}\|D^{s}_h\nabla^2\eta\|^2_{L^2}-C(\|\eta\|_{H^2(\omega)}).
$$
Similarly, since $\|D^{s}_{-h}D^{s}_h\eta\|_{L^{\infty}(\omega)}\leq \|\eta\|_{H^2(\omega)}$ uniformly,  we estimate
$$
|a^2_b(\eta,\nabla^2\eta,D^{s}_{-h}D^{s}_h\eta)|\leq C(\|\eta\|_{H^2(\omega)}).
$$
To estimate $a^3_b$ we first notice that $\|P_0(\eta,\nabla\eta)\|_{L^2}\leq C\|\eta\|_{L^{\infty}}\|\nabla\eta\|_{L^4}^2\leq C(\|\eta\|_{H^2(\omega)})$. Moreover,
$$
\|P'_0(\eta,\nabla\eta)D^{s}_h\eta\|_{L^2}\leq C\|\eta\|_{L^{\infty}}\|\nabla\eta\|_{L^4}\|\nabla D^{s}_h\eta\|_{L^4}\leq C(\|\eta\|_{H^2(\omega)}).
$$
Now we can use integration by parts and Young's inequality in the same way as in the estimate for $a^1_b$ to get
$$
|a^3_b(\eta,\nabla\eta,\nabla^2\eta;D^{1/2}_{-h}D^{1/2}_h\eta,\nabla,D^{1/2}_{-h}D^{1/2}_h\eta)
\leq \frac{C}{8}\|D^{s}_h\nabla^2\eta\|^2_{L^2}+C(\|\eta\|_{H^2(\omega)}).
$$
Estimate for $a^4_b$ is done in analogous way by integration by parts and using:
$$
\|D^s_h P_0(\eta,\nabla\eta)\|_{L^2}\leq \|\eta\|_{L^{\infty}}\|\eta\|_{W^{1,4}}\nabla D^{s}_h\eta\|_{L^4}\leq C(\|\eta\|_{H^2(\omega)}).
$$
Hence,
$$
|a^4_b(\eta,\nabla\eta;\nabla D^{s}_{-h}D^{s}_h\eta,\nabla^2 D^{s}_{-h}D^{s}_h\eta)|
\leq \frac{C}{8}\|D^{s}_h\nabla^2\eta\|^2_{L^2}+C(\|\eta\|_{H^2(\omega)}).
$$
Finally, the last term $a^5_b$ is a lower order term and is easily estimated using the same inequalities:
$$
|a^5_b(\eta,\nabla\eta,D^{s}_{-h}D^{s}_h\eta,\nabla D^{s}_{-h}D^{s}_h\eta)|\leq C(\|\eta\|_{H^2(\omega)}).
$$
\end{proof}
\subsection{Closing the estimates--Proof of Theorem~\ref{RegTem}}
In this section we finish the proof of Theorem \ref{RegTem}. Please observe first, that due to the Sobolev embedding theorem and due to the trace theorem~\cite[Lemma~2.4]{BreitSchwarzacher} we find for all $\theta\in (0,1)$ and all $s\in (0,\frac12)$
\[
\norm{\eta}_{L^\infty(0,T;C^{0,\theta}(\omega))}\leq c\norm{\eta}_{L^\infty(0,T;H^2(\omega))}\text{ and } \norm{\partial_t\eta}_{L^2(0,T;H^s(\omega))}\leq c\norm{\bu}_{L^2(0,T;H^1(\Omega_{\eta}))}.
\]
Assume that $s\in (0,\frac12)$ and take 
\[
\Big (\test{\big (D^s_{-h}D^s_h\eta -\cor(D^s_{-h}D^s_h\eta)\big )},D^s_{-h}D^s_h\eta -\cor(D^s_{-h}D^s_h\eta)\Big)
\]
 as a test function in \eqref{WeakSolEq} and integrate from $0$ to $T$. The test function is admissible by construction, see Proposition \ref{TestFunction}. The estimates on the forms $a_m$ and $a_b$ connected to the elastic energy follow directly by Lemma~\ref{StructureRegEstimate}. Indeed, since $\cor(D^s_{-h}D^s_h\eta)$ is constant in space direction and hence does not change the estimate on the derivatives of $\eta$ we find (using the uniform bounds on $\lambda_\eta$) that
 $$
\inf_{\omega}{(\gamma^2(\eta))}\, a_b(t,\eta,(D^{-s}_hD^{s}_h\eta-\cor(D^s_{-h}D^s_h\eta))\geq \|D^{s}_h\nabla^2\eta\|_{L^2(\omega)}-C(\|\eta\|_{H^2(\omega)}).
$$
Hence we are left to estimate the term coming from the structure inertia.
Using partial integration and Corollary~\ref{cor:cor}, we find
\begin{align*}
&\int_0^T\absBB{\int_{\omega}\partial_t\eta\partial_t(D^s_{-h}D^s_h\eta -\cor(D^s_{-h}D^s_h\eta)\, dy}\, dt
=\int_0^T\absBB{\int_{\omega}\partial_t\eta D^s_{-h}(\partial_t(D^s_h\eta -\cor(D^s_h\eta))\, dy}\, dt
\\
&\quad =\int_0^T\absBB{\int_{\omega}(D^s_{h}\partial_t\eta)^2- D^s_h\partial_t\eta \partial_t\cor(D^s_h\eta)\, dy}\, dt
\\
&\quad  \leq c\|\partial_t\eta\|^2_{L^2(0,T;H^s(\omega))} + \int_0^T\norm{\partial_t\eta}_{W^{1,s}(\omega))}(\norm{\partial_t\eta}_{W^{1,s}(\omega))}+\norm{\partial_t\eta}_{L^2(\omega))}\norm{\nabla\eta}_{L^2(\omega)})\, dt
\\
&\quad \leq c\|\partial_t\eta\|^2_{L^2(0,T;H^s(\omega))} + cT\norm{\partial_t\eta}_{L^\infty(0,T;L^2(\omega))}^2\norm{\nabla\eta}_{L^\infty(0,T;L^2(\omega))}^2
\\
&\quad \leq c\|\bu\|^2_{L^2(0,T;H^1(\Omega_{\eta}(t))}+ cT\norm{\partial_t\eta}_{L^\infty(0,T;L^2(\omega))}^2\norm{\nabla\eta}_{L^\infty(0,T;L^2(\omega))}^2\leq c C_0^2.
\end{align*}
Here in the last estimate we used the trace theorem~\cite[Lemma~2.4]{BreitSchwarzacher} and the coupling condition \eqref{kinematic}. Notice that this term cannot be estimated in a purely hyperbolic problem and that here it is essential to use the coupling and the fluid dissipation.

Let us next prove the estimates related to the fluid part. From Proposition~\ref{TestFunction} and the energy inequality \eqref{EI-weak} we have the following estimate

\begin{align*}
&\|\nabla\test{(D^s_{-h}D^s_h\eta -\cor(D^s_{-h}D^s_h\eta))}\|_{L^{\infty}(0,T;L^2(\Omega_\eta(t))}
\\
&\quad 
\leq C\big (\|D^s_{-h}D^s_h\eta\|_{L^{\infty}(0,T;H^1(\omega))}+
\|(D^s_{-h}D^s_h\eta)\nabla\eta\| _{L^{\infty}(0,T;L^2(\omega))} \big )
\\
&\quad \leq C\big (
\|\eta\|_{L^{\infty}(0,T;H^2(\omega))}
+\|D^s_{-h}D^s_h\eta\|_{L^{\infty}(0,T;L^{\infty}(\omega))}
\|\nabla\eta\| _{L^{\infty}(0,T;L^2(\omega))}
\big )
\\
&\quad \leq C\big (
\|\eta\|_{L^{\infty}(0,T;H^2(\omega))}+\|\eta\|_{L^{\infty}(0,T;H^2(\omega))}^2
\big )
\leq C(C_0+C_0^2).
\end{align*}
This allows to estimate the integrals:
$$
\absBB{
\int_0^T\int_{\Omega_{\eta}(t)}(-\bu\otimes\bu:\nabla\test{(D^s_{-h}D^s_h\eta -\cor(D^s_{-h}D^s_h\eta))}
+\sym\nabla\bu:\sym\nabla\test{D^s_{-h}D^s_h\eta})
\, dx}
$$
$$
\leq \|\nabla\test{D^s_{-h}D^s_h\eta}\|_{L^{\infty}(0,T;L^2(\Omega_\eta(t))}(\|\nabla\bu\|^2_{L^2(0,T;L^2(\Omega_{\eta}(t)}+\|\nabla\bu\|_{L^2(0,T;L^2(\Omega_{\eta}(t)}))
\leq C(C_0+C_0^2)^2.
$$
The most difficult estimate is the estimate involving the distributional time-derivative of $v$. It can be estimated using Lemma~\ref{lem:IbP}; indeed by defining $p=2=p'$ and $\tilde{a}=6$ we get that $\frac{\tilde{a} p}{\tilde{a}p-\tilde{a}-p}=3$. Hence using the fact that 
$\frac{1}{2}+\frac{1}{3}=\frac{5}{6}$ and
\[
\frac{3}{2} -\frac{2}{3}=\frac{5}{6}<1=2-\frac{2}{2}\text{ and so }W^{\frac{3}{2},3}(\omega)\subset W^{2,2}(\omega),
\]
we find by H\"older's inequality and Sobolev embedding that for every $\theta\in (0,1)$ there is a constant $c$, such that
\begin{align*}
(I)=&\absBB{\int_0^T\int_{\Omega_{\eta}(t)}\bu\cdot\partial_t\test(D^s_{-h}D^s_h\eta-\cor(D^s_{-h}D^s_h\eta))\, dx}
\\
&\leq c\Big(h^{\theta-s}+\|D^s_h\eta\|_{L^\infty(0,T;W^{1,3}(\omega))}\Big)\|\bu\|_{L^{2}(0,T;W^{1,2}(\Omega_{\eta}(t))}\|\partial_tD^s_h\eta\|_{L^2(0,T;L^2(\omega))}
\\
&\quad 
+c\|\bu\|_{L^{2}(0,T;W^{1,2}(\Omega_{\eta}(t))}\norm{D^s_{-h}D^s_h\eta \partial_t\eta}_{L^2(0,T;L^{6/5}(\omega))}.
\end{align*}
Hence choosing $\theta=s$, we find
\begin{align*}
(I)& \leq c\Big(1+\|\eta\|_{L^\infty(0,T;W^{3/2,3}(\omega))}\Big)\|\bu\|_{L^{2}(0,T;W^{1,2}(\Omega_{\eta}(t))}\|\partial_tD^s_h\eta\|_{L^2(0,T;L^2(\omega))}
\\
&\quad 
+c\|\bu\|_{L^{2}(0,T;W^{1,2}(\Omega_{\eta}(t))}\norm{\partial_t\eta}_{L^\infty(0,T;L^2(\omega)}\norm{D^s_{-h}D^s_h\eta}_{L^2(0,T;L^3(\omega)}
\\
&\quad \leq c\Big(1+\|\eta\|_{L^\infty(0,T;W^{2,2}(\omega))}\Big)\|\bu\|_{L^{2}(0,T;W^{1,2}(\Omega_{\eta}(t))}\|\partial_tD^s_h\eta\|_{L^2(0,T;L^2(\omega))}
\\
&\qquad 
+c\|\bu\|_{L^{2}(0,T;W^{1,2}(\Omega_{\eta}(t))}\norm{\partial_t\eta}_{L^\infty(0,T;L^2(\omega)}\norm{\eta}_{L^2(0,T;W^{1,3}(\omega))}
\\
&\quad \leq cC_0 (C_0+C_0^2),
\end{align*}
and the estimate on the term of the time-derivative is complete. The result follows by combining the obtained estimates.
%
%

\section{Compactness rewritten}
\label{sec:auba}
We introduce the following version of the celebrated Aubin-Lions compactness lemma~\cite{aubin1963theoreme,lions1969quelques}. 
The version below is tailored to be applicable for the coupled systems of PDE like the fluid-structure interaction which we we study in this paper. The key point is to fully decouple the {\em compactness assumption} in space and the {\em compactness assumption} in time. We emphasize this fact by showing that under appropriate conditions the product of two weak convergent sequences decouple in the limit. It in some sense unifies ideas from time-space decoupling with compensated compactness approaches of div-curl type (see e.g.~\cite{BulBurSch18} for some further discussion on that matter).

In this context, the most difficult property to capture is the compactness in time assumption. Commonly it is given in the form of a uniform bound on time-derivative is certain dual space or more precisely a uniform continuity assumption in time. What turned out to be the key observation is that it suffices only to extract the uniform continuity properties over a suitable approximation of its argument. In the theorem below requirement (3) summarizes the time-compactness assumption. As can be seen, no function space is appearing. The assumption is that the pairing of the continuity in time for $g_n$ is uniform with respect to a given suitable approximation of $f_n$.
 
This non-function space type requirement is necessary for the application in the context of fluid-structure interactions. Indeed, the weak time derivative of an approximate sequence $\partial_t\eta_\varepsilon,v_\varepsilon$ is defined merely over a non-linear coupled space that changes both with respect to time and with respect to the approximation parameter $\epsilon$ itself.
\begin{theorem}
\label{thm:auba}
{Let $X,Z$ be two Banach spaces, such that $X'\subset Z'$.
Assume that $f_n:(0,T)\to X$ and $g_n: (0,T)\to X'$}. Moreover assume the following: 
\begin{enumerate}
\item The {\em weak convergence}: for some $s\in [1,\infty]$ we have that $f_n\toweakstar f$ in $L^s(X)$ and $g_n\toweakstar g$ in $L^{s'}(X')$.
\item The {\em approximability-condition} is satisfied: For every $\delta\in (0,1]$ there exists a $f_{n,\delta}\in L^s(0,T;X)\cap L^1(0,T;Z)$,
%
such that for every $\epsilon\in (0,1)$ there exists a $\delta_\epsilon\in(0,1)$ (depending only on $\epsilon$) such that 
\[
\norm{f_n-f_{n,\delta}}_{L^s(0,T;X)}\leq \epsilon\text{ for all } \delta\in (0,\delta_\epsilon]
\]
and for every $\delta\in (0,1]$ there is a $C(\delta)$ such that
\[
\norm{f_{n,\delta}}_{L^1(0,T;Z)}\, dt\leq C(\delta).
\]
Moreover, we assume that for every $\delta$ there is a function $f_\delta$, and a subsequence such that $f_{n,\delta}\toweakstar f_\delta$ in $L^s(X)$.

\item The {\em equi-continuity} of $g_n$. For every $\epsilon>0$ and $\delta>0$ that there exist a $n_{\epsilon,\delta}$ and a $\tau_{\epsilon,\delta}>0$, such that for all $n\geq n_{\eps,\delta}$  and all $\tau\in (0,\tau_{\eps,\delta}]$
\[
\int_0^{T-\tau} \absB{\dashint_{0}^\tau\skp{g_n(t)-g_n(t+s),f_{n,\delta}(t)}_{X',X}\, ds}\, dt\leq \epsilon.
\]

\item The {\em compactness assumption} is satisfied: $X'\hookrightarrow \hookrightarrow Z'$. More precisely, every uniformly bounded sequence in $X'$ has a strongly converging sub-sequence in $Z'$. 
\end{enumerate}
Then there is a subsequence, such that
\[
\int_0^T\skp{f_n,g_n}_{X,X'}\, dt\to \int_0^T\skp{f,g}_{X,X'}\, dt.
\]
\end{theorem}
\begin{remark}[Modification for applications]
\label{rem:modification}
With regard of our application it seems somehow natural to replace (3) by the following condition
\begin{enumerate}
\item[(3')] 
The {\em equi-continuity} of $g_n$. We require that there exists an $\alpha\in (0,1]$ {a sequence $A_n$ that is uniformly bounded in $L^1([0,T])$}, such that for every $\delta>0$ that there exist a $C(\delta)>0$ and an $n_\delta\in \N$ such that for $\tau>0$ and a.e.\ $t\in [0,T-\tau]$
\[
\sup_{n\geq n_{\delta}} \absB{\dashint_{0}^\tau\skp{g_n(t)-g_n(t+s),f_{n,\delta}(t)}_{X',X}\, ds} \leq C(\delta)\tau^\alpha(A_n(t)+1).
\]
\end{enumerate}
Here (3') implies (3) by integration over $[0,T-\tau]$ and an appropriate choice of $\tau_{\delta,\epsilon}$.
\end{remark}

\begin{remark}[Classic Aubin-Lions lemma]
Let us explain how Theorem \ref{thm:auba} relates to the classic Aubin-Lions lemma. The simplest case is when $Z$ is a compact subspace of $X$, $f\in L^2(Z)$ and $\partial_tg_n\in L^2(Z')$. In this case one may take $f_{n,\delta}=f_n$ and finds for $s<t$
\[
\abs{\skp{g_n(t)-g_n(s),f_m(t)}}=\Big|\int_{s}^t\skp{\partial_t g_n(\tau)}{f_m(t)}\, d\tau
\Big|\leq \norm{f_m(t)}_Z\int_{s}^t\norm{\partial_t g_n(\tau)}_{Z'}\, d\tau\leq c\abs{t-s}^\frac{1}{2}\norm{f_m(t)}_Z,
\]
with $c=\norm{\partial_t g_n}_{L^2(0,T;Z')}$.
The classic Gelfand triple is then the particular case when $X$ is a Hilbert space. Since then $Z\subset \subset X \subset Z'$ implies the same argument as above. 

The generalization to allow that $Z$ is independent of the regularity of $f_n$ is essentially some hidden interpolation result (also known as Ehrling property). Here classically one can use convolution estimates to show that a mollifier in one space is uniformly close, while in the other (smaller space) they are merely bounded. One standard example is the periodic solutions over the torus $Q$ and $X=H^a_{per}(Q)$ and $Z=H^c_{per}(Q)$, such that $f_n\in H^b_{per}(Q)$ uniformly with $a<b<c$. Then convolution with the standard mollifying kernel $\psi_\delta$ implies that $\norm{f-f*\psi_\delta}_{H^r_{per}(Q)}\leq C\delta^{s-r}\norm{f}_{H^s_{per}(Q)}$ which implies precisely the wanted properties in (2) above. This shows that condition (2) can be seen as a "spatial compactness" condition in the Aubin-Lions lemma (or more generally in Simon's compactness theorem \cite{Simon}). The condition (3) could be viewed as a "temporal compactness", i.e.\ as equi-continuity of time shifts in the weaker space $Z'$.
\end{remark}

\begin{remark}[Function spaces]
Since we do not assume that $Z$ is dense in $X$, an additional clarification of condition (5) is required. Let $\overline{Z}^{\norm{\cdot}_X}{=\overline{X\cap Z}^{\norm{\cdot}_X}}$ be the closure of $Z$ w.r.t.\ $X$ and $(\overline{Z}^{\norm{\cdot}_X})'$ its dual (w.r.t.\ $X$ as pivot space). Then condition (5) has a meaning in the following sense $(\overline{Z}^{\norm{\cdot}_X})'\hookrightarrow \hookrightarrow Z'$.
\end{remark}

\begin{proof}[Proof of Theorem~\ref{thm:auba}]
In this prove we will produce for every $\epsilon>0$ an $n_\epsilon\in \N$, a $\tau_\epsilon>0$ with $\tau_\epsilon\to 0$ for $\epsilon\to 0$ and a subsequence of $\skp{f_n,g_n}_{X,X'}$ such that 
\[
\absB{\int_0^{T-\tau_\epsilon}\skp{f_n,g_n}_{X,X'}-\skp{f_m,g_m}_{X,X'}\, dt}\leq \epsilon\]
for all $n,m\geq n_\epsilon$. This then allows to construct the desired converging subsequence by taking a discrete sequence $\epsilon_i\to 0$ and a respective diagonal argument. 

Hence let $\epsilon>0$, we may choose $\delta_\epsilon$ in such a way, that for all $\delta\in (0,\delta_\epsilon]$,
\begin{align}
\label{eq:approx}
\norm{f_n-f_{n,\delta}}_{L^s(0,T;X)}\leq \epsilon
\end{align}
 Next we fix $\tau_{\epsilon,0}>0$ and $n_{\epsilon,0}$, such that for all $\tau\in (0,\tau_{\epsilon,0}]$
\begin{align}
\label{eq:4}
\sup_{n\geq n_{\epsilon,0}}\int_0^{T-\tau} \absB{\dashint_0^\tau\skp{g_n(t)-g_n(s),f_{n,\delta}(t)}_{X',X}\, ds}\, dt\leq \epsilon.
\end{align}
Fix $N\in \N$ such that $\tau_\epsilon:=\frac{T}{N}\leq \tau_{\epsilon,0}$. 
For $k\in\{0,...,N-1\}$ and $n\in \N$ we define
\[
g_n^k=\dashint_{k\tau}^{(k+1)\tau}g_n(s)\, ds.
\]
 This implies by Jensen's inequlaity 
\[
\norm{g_n^k}_{X'}\leq \dashint_{k\tau}^{(k+1)\tau}\norm{g_n(s)}_{X'}\, ds,
\]
and so we define for the given $\tau_\epsilon$
\[
g_n^{\tau_\epsilon}(t):=g_n^k\text{ for }t\in [k\tau_\epsilon,(k+1)\tau_\epsilon).
\]
Since
\[
\sup_{[k\tau_\epsilon,(k+1)\tau_\epsilon]\subset[0,T]}\sup_{n\in \N}\norm{g_n^k)}_{X'}\leq \frac{C}{\tau_\epsilon},
\]
we find
by the {\em compactness assumption,} that we can find a subsequence for which there exists a $n_{\epsilon,1}$, such that
\begin{align}
\label{eq:comp}
\sup_{k}\norm{g_n^k-g_m^k}_{Z'}\leq \epsilon_0\text{ for all }n,m>n_{\epsilon,1}.
\end{align}
In particular there exists $g^{\tau_\epsilon}$ and a subsequence, such that$g_n^{\tau_\epsilon}\to g^{\tau_\epsilon}$ strongly in $L^\infty(0,T;Z')$. Clearly, by the uniform bounds we find that $g^\tau\in L^{s'}(X')$.

At this point $\tau_\epsilon$ and $\delta_\epsilon$ are fixed. Hence we may define 
\[
\epsilon_0:=\frac{\epsilon}{C(\delta_{\epsilon})},
\]
where $C(\delta_{\epsilon})$ is defined via (2). Therefore, we find an $n_\epsilon\in \N$, such that for all $n,m\geq n_{\epsilon}$
\begin{align}
\label{eq:3}
\absB{\int_0^T\skp{f_{n,\delta_\epsilon},g_n^{\tau_\epsilon}-g_m^{\tau_\epsilon}}\, dt}_{X,X'}\leq \norm{f_{n,\delta_\epsilon}}_{L^1(0,T;Z)}\norm{g_n^{\tau_\epsilon}-g_m^{\tau_\epsilon}}_{L^\infty(0,T;Z'}\leq \epsilon.
\end{align}
Now all preparations have been made in order to estimate:
\begin{align*}
&\absB{\int_0^{T-\tau_\epsilon}\skp{f_n(t),g_n(t)}_{X,X'}-\skp{f_m(t),g_m(t)}_{X,X'}\, dt}
\\
&\quad\leq \absB{\int_0^{T-\tau_\epsilon}\skp{f_{n,\delta_\epsilon}(t),g_n(t)}_{X,X'}-\skp{f_{m,\delta_\epsilon}(t),g_m(t)}_{X,X'}}
\\
&\qquad+ \absB{\int_0^{T-\tau_\epsilon}\skp{f_n(t)-f_{n,\delta_\epsilon}(t),g_n(t)}_{X,X'}-\skp{f_m(t)-f_{m,\delta_\epsilon}(t),g_m(t)}_{X,X'}\,dt}
\\
&
\quad \leq \absB{\int_0^{T-\tau_\epsilon}\skp{f_{n,\delta_\epsilon}(t),g_n(t)}_{X,X'}-\skp{f_{m,\delta_\epsilon}(t),g_m(t)}_{X,X'}}+2\epsilon.
\end{align*}
We estimate the left 
\begin{align*}
(I):&=\int_0^{T-\tau_\eps}\skp{f_{n,\delta_\epsilon}(t),g_n(t)}_{X,X'}-\skp{f_{m,\delta_\epsilon}(t),g_m(t)}_{X,X'}\, dt
\\
&=\int_0^{T-\tau_\eps}\skp{f_{n,\delta_\epsilon}(t)
,g_n(t)-g_n^{\tau_\eps}(t)}_{X,X'}\, dt
 +
\int_0^{T-\tau_\eps}\skp{f_{n,\delta_\epsilon}(t),g_n^{\tau_\eps}(t)-g_n^{\tau_\eps}(t)}_{X,X'}\, dt
\\
&\quad +\int_0^{T-\tau_\eps}\skp{f_{m,\delta_\epsilon}(t)-f_{n,\delta_\epsilon}(t),g_m^{\tau_\eps}(t)}_{X,X'}\, dt
-\int_0^{T-\tau_\eps}\skp{f_{m,\delta_\epsilon}(t),g_m(t)-g_n^{\tau_\eps}(t)}_{X,X'}\, dt
\\
&=\sum_{k=0}^{N-2}\int_{k\tau_\eps}^{(k+1)\tau_\eps}\dashint_0^{\tau_\eps}\skp{f_{n,\delta_\epsilon}(t)
,g_n(t)-g_n(s)}_{X,X'}\, ds\, dt
 +
\sum_{k=0}^{N-2}\int_{k\tau_\eps}^{(k+1)\tau_\eps}\skp{f_{n,\delta_\epsilon}(t),g_n^k-g_m^k}_{X,X'}\, dt
\\
&\quad +\sum_{k=0}^{N-2}\int_{k\tau_\eps}^{(k+1)\tau_\eps}\skp{f_{m,\delta_\epsilon}(t)-f_{n,\delta_\epsilon}(t),g_m^k}_{X,X'}\, dt
-\sum_{k=0}^{N-2}\int_{k\tau_\eps}^{(k+1)\tau_\eps}\dashint_0^{\tau_\eps}\skp{f_{m,\delta_\epsilon}(t),g_m(t)-g_m(s)}_{X,X'}\, ds\, dt
\\
&= (II)+(III)+(IV)+(V).
\end{align*}
First observe that $(II)$ and $(V)$ can be estimated using the {\em equi-continuity condition}, namely \eqref{eq:4}. Term $(III)$ is estimated using the {\em compactness condition}, namely \eqref{eq:3}. Finally for $(IV)$ we deviate and apply \eqref{eq:3} a second and third time
\begin{align*}
(IV)&=\int_0^T\skp{f_{m,\delta_\epsilon}-f_{n,\delta_\epsilon},g_m^{\tau_\epsilon}}_{X,X'}\, dt
\\
&=\int_0^T\skp{f_{m,\delta_\epsilon}-f_{n,\delta_\epsilon},g^{\tau_\epsilon}}_{X,X'}+\skp{f_{m,\delta_\epsilon}-f_{n,\delta_\epsilon},g^{\tau_\epsilon}_m-g^{\tau_\epsilon}}_{X,X'}\, dt
\\
&\leq 
\absB{\int_0^T\skp{f_{m,\delta_\epsilon}-f_{n,\delta_\epsilon},g^{\tau_\epsilon}}_{X,X'}} + 2\epsilon.
\end{align*}
Now we take another subsequence of $f_{n,\delta_\epsilon}\in L^{s}(X)$ 
that converges weakly*. Hence we may eventually increase $n_\epsilon$ one last time (in dependence of $g^{\tau_\epsilon}$) and find that for this subsequence and $n,m\geq n_\epsilon$
\[
\abs{(IV)}\leq 3\epsilon.
\]
This finishes the proof.
\end{proof}

\section{The existence result}
\label{sec:thm1}
\subsection{The approximate system}
In this section we construct approximate solutions $(\bu^{\varepsilon},\eta^{\varepsilon})\in \mathcal{V}_S$, $\eta^{\varepsilon}\in L^{\infty}(0,T;H^3(\omega))$ which satisfy the following weak formulation:
\begin{equation}\label{WeakSolEqReg}
\begin{array}{c}
\displaystyle{\frac{d}{dt}\int_{\Omega_{\eta}(t)}\bu^{\varepsilon}\cdot\bq\, dx+\int_{\Omega_{\eta}(t)}\Big (-\bu^{\varepsilon}\cdot\partial_t\bq-\bu^{\varepsilon}\otimes\bu^{\varepsilon}:\nabla\bq+\sym\nabla\bu^{\varepsilon}:\sym\nabla\bq\Big )\, dx}
\\
\displaystyle{+\frac{d}{dt}\int_{\omega}\partial_t\eta^{\varepsilon} \xi\, dy-\int_{\omega}\partial_t\eta^{\varepsilon}\partial_t \xi
\, dy+a_m(t,\eta^{\varepsilon},\xi)+a_b(t,\eta^{\varepsilon},\xi)
+\varepsilon\int_{\omega}\nabla^3_x\eta^{\varepsilon}:\nabla^3_x \xi\, dy=0,}
\end{array}
\end{equation}
where $\varepsilon>0$ is a regularizing parameter and with initial conditions $\eta_0,\eta_1,\bu_0$.
In this section we prove the following Theorem:
\begin{theorem}\label{ApproxSolution} There exists a $T>0$ just depending on $\partial \Omega$ and the initial data, such that for every $\varepsilon\in(0,1]$ there exists a weak solution $(\bu^{\varepsilon},\eta^{\varepsilon})$ to the regularized problem~\eqref{WeakSolEqReg}. Moreover, the weak solution satisfies the following uniform in $\varepsilon$ estimate: 
\begin{equation}\label{UniformEnergyEps}
\|\bu^{\varepsilon}\|_{\mathcal{V}_F}+\|\eta^{\varepsilon}\|_{\mathcal{V}_S}
+\|\eta^{\varepsilon}\|_{L^2(0,T;N^{s,2}(\omega))}\leq C,
\end{equation}
for every (fixed) $s<\frac{5}{2}$, with $C$ just depends on $\partial \Omega$ and the initial conditions.
\end{theorem}
The existence of regularized solutions can be proved following the ideas and techniques introduced in \cite{BorSunnon-linear}. The problem solved in~\cite{BorSunnon-linear} is actually very similar to the regularized system above since there the existence of a solution to a FSI problem with a structure being an elastic shell with a non-linear Koiter membrane energy without bending energy, but with a (linear) regularization term of fourth order is shown. In order to be able to treat the non-linear bending energy in an analogous way we have to include a sixth order regularization term. Another difference comes from the fact that in~\cite{BorSunnon-linear} cylindrical geometry is considered. Nevertheless the introduced existence scheme does not depend on the geometry of the problem and more general geometries can be handled by combining the existence proof with the estimates in this paper and in~\cite{LenRuz}. To avoid lengthy repetitions of the arguments analogous to \cite{BorSunnon-linear} here, we summarize the main steps of the construction of a weak solutions with emphasis on the differences coming from the non-linear bending term and the setting of more general geometries. The main steps of the construction are:
\begin{enumerate}
\item {\bf Arbitrary Lagrangian-Eulerian (ALE) formulation}. We reformulate the problem in a fixed reference domain $\Omega$ using suitable change of variables. This approach is popular in numerics and the change of variables is called Arbitrary Lagrangian-Eulerian (ALE) mapping. The formulation in the fixed reference domain is called ALE formulation of the FSI problem. We use the mapping $\bpsi_{\eta}$ (introduced in Definition~\ref{def:coor}) as an ALE mapping.
\item {\bf Construction of the approximate solutions}. We construct the approximate solutions using time-discretizations and operator splitting methods. We use the Lie splitting strategy (also known as Marchuk-Yanenko splitting) to decouple the FSI problem.
\item {\bf Uniform estimates}. Let $\dt>0$ be the time-discretization parameter. We show that the constructed approximate solutions satisfy uniform bounds w.r.t. $\dt$ (and $\varepsilon$) in the energy function spaces. We identify weak and weak* limits.
\item {\bf Compactness}. We prove that the set of approximate solutions is compact in suitable norms. By using the compactness we prove that a limit of the sequence of approximate solutions is a weak solution to the regularized FSI problem. Here we use a generalization of the Aubin-Lions-Simon lemma for discrete in time solutions adapted to the moving domain problems from~\cite{BorSunComp17}.
\end{enumerate}
\noindent
Since a solution is constructed by decoupling the problem, the largest difference from \cite{BorSunnon-linear} is in the second step where in the structure sub-problem we include also the non-linear bending energy. However, we will show that the bending term can be discretized in an analogous way as the membrane term. Other steps are analogous as in~\cite{BorSunnon-linear} using the sixth order regularization. For the convenience of the reader we will describe the details of the time-discretization of the structure sub-problem with the corresponding uniform estimates in the time-discretization parameter $\dt$. We conclude this chapter with the description of the compactness step. Generally, for more details on the procedure we refer the reader to~\cite{BorSunnon-linear}. 

In the rest of the subsection we fix the regularizing parameter $\varepsilon$ and drop superscripts $\varepsilon$ in $(\bu^{\varepsilon},\eta^{\varepsilon})$ since there is no chance of confusion.

\smallskip
{\bf Construction of discrete approximations}

\noindent
The main problem in the construction of approximate solution is how to discretize the Fr\' echet derivatives of $\mathbf{G}$ and $\mathbf{R}$ to obtain the discrete analogue of $\mathbf{R}'(\eta)\partial_t\eta=\partial_t\mathbf{R}(\eta)$. In \cite{BorSunnon-linear} this was achieved by using the fact that the first fundamental form was polynomial of order two of $\eta$ and $\nabla\eta$ which was a consequence of the cylindrical geometry. Here we consider a more general geometry so we need to develop a more general approach.

For a given end-time $T$, we fix $\dt$ as the time step, such that $[0,T]=[0,N\dt]$ for some $N\in \Nbb$. Now let $(\eta^n)_{n=1}^N$ be a given time-discrete solution and $\tilde{\eta}$ be the piece-wise linear function in time such that $\tilde{\eta}(n\dt)=\eta^n$.
Then we have
$$
\mathbf{R}'(\tilde{\eta})\frac{\eta^{n+1}-\eta^n}{\dt}
=\mathbf{R}'(\tilde{\eta})\partial_t\tilde{\eta}
=\partial_t\mathbf{R}(\tilde{\eta})
\;{\rm on}\; [n\dt,(n+1)\dt].
$$
Notice that the expression $\mathbf{R}'(\tilde{\eta})\partial_t\tilde{\eta}$ is a third order polynomial in the $t$ variable so we can compute its integral $\int_{n\dt}^{(n+1)\dt}$ by using Newton-Cotes formula. Hence, by defining $\overline\eta^{n+1} := \frac{\eta^{n+1} + \eta^n}{2}$ we find the approximation of ${\bf G}'(\eta)\xi$ and ${\bf R}'(\eta)\xi$ in the following way:
\begin{equation}\label{GDis}
{\bf G}'(\eta^{n+1},\eta^n)\xi:=\frac{1}{3}\Big ({\bf G}'(\eta^n)+4{\bf G}'(\overline{\eta}^{n+1})+{\bf G}^{n+1}\Big )\xi
\end{equation}
and 
\begin{equation}\label{RDis}
{\bf R}'(\eta^{n+1},\eta^n)\xi:=\frac{1}{3}\Big ({\bf R}'(\eta^n)+4{\bf R}'(\overline{\eta}^{n+1})+{\bf R}^{n+1}\Big )\xi.
\end{equation}
By straightforward calculation it follows that
\begin{align*}
{\bf G}'(\eta^{n+\frac{1}{2}},\eta^n)\frac{\eta^{n+\frac 1 2}-\eta^n}{\Delta t}
&=\dt\int_{n\dt}^{(n+1)\dt}\frac{d}{dt}\mathbf{G}(\tilde{\eta})
\\
&=\frac{1}{\Delta t}({\bf G}(\eta^{n+\frac 1 2})-{\bf G}(\eta^n))
\end{align*} 
which is the correct substitute for ''$\partial_t{\bf G}(\eta)={\bf G}'(\eta)\partial_t\eta$''.
Analogously we find as substitute for ''$\partial_t{\bf R}(\eta)={\bf R}'(\eta)\partial_t\eta$''
$$
{\bf R}'(\eta^{n+\frac{1}{2}},\eta^n)\frac{\eta^{n+\frac 1 2}-\eta^n}{\Delta t}=\frac{1}{\Delta t}({\bf R}(\eta^{n+\frac 1 2})-{\bf R}(\eta^n)).
$$
These identities will be used to derive a semi-discrete uniform energy inequality. First we define the sequence of approximate solutions by solving the following problems.

\smallskip
{\bf Structure sub-problem}. 
\\
Find $(v^{n+\frac 1 2},\eta^{n+\frac 1 2})\in \big (H^2_0(\omega)\cap H^3(\omega)\big )^2$ such that:
\begin{align}
\begin{aligned}
&\int_{\omega}\frac{\eta^{n+\frac 1 2}-\eta^{n}}{\Delta t}\phi\, dy=\int_{\omega}v^{n+\frac 1 2}\phi\, dy,
\\ 
&\int_{\omega}\frac{v^{n+\frac 1 2}-v^{n}}{\Delta t}\psi\,dy
+\frac{1}{2}\int_{\omega}{\mathcal A}{\bf G}(\eta^{n+\frac 1 2}):{\bf G}'(\eta^{n+\frac 1 2},\eta^n)\psi\, dy
\\
&\quad +\frac{1}{24}\int_{\omega}{\mathcal A}{\bf R}(\eta^{n+\frac 1 2}):{\bf R}'(\eta^{n+\frac 1 2},\eta^n)\psi\, dy
+\varepsilon\int_{\omega}\nabla^3\eta^{n+\frac 1 2}\nabla^3\psi\, dy=0,
\end{aligned}
\label{DProb3}
\end{align}
for all $(\phi,\psi)\in L^2(\omega)\times \big (H^2_0(\omega)\cap H^3(\omega)\big)$. 

The existence of a solution to the above problem follows by Schaefer's Fixed Point Theorem as it was demonstrated in~\cite[Proposition 4]{BorSunnon-linear}).

\smallskip
{\bf Fluid sub-problem}.
\\
The fluid problem stays the same as in~\cite{BorSunnon-linear} (which is the advantage of the operator splitting method). Since the domain deformation is calculated in the structure sub-problem and does not change in the fluid sub-problem we set $\eta^{n+1}=\eta^{n+\frac12}$, and define $({\bf u}^{n+1},v^{n+1})\in {\mathcal V}_F^{\eta^n}\times L^2 (\omega)$ by requiring that
for all $({\bf q},\xi)\in{\mathcal V}_F^{\eta^n}\times L^2 (\omega)$ such that
${\bf q}_{|\Gamma}=\xi\bn$, the following weak formulation holds: 
\begin{align*}
\label{D1Prob1}
\begin{aligned}
&{}\int_{\Omega}  J^n \bigg( \frac{{\bf u}^{n+1}-{\bf u}^{n}}{\Delta t}\cdot{\bf q}+
\frac 1 2\left[({\bf u}^n-{\bf w}^{n+\frac 1 2})\cdot\nabla^{\eta^n}\right]{\bf u}^{n+1}\cdot{\bf q}
\\
&\quad
 -\frac 1 2 \left[({\bf u}^n-{\bf w}^{n+\frac 1 2})\cdot\nabla^{\eta^n} \right] {\bf q}\cdot{\bf u}^{n+1}\bigg)\, dx
\\
&\quad 
+\frac{1}{2} \int_{\Omega}\frac{J^{n+1}-J^n}{\dt}{\bf u}^{n+1}\cdot{\bf q}\, dx +2\int_{\Omega}J^n{\bf D}^{\eta^n}({\bf u}):{\bf D}^{\eta^n}({\bf q})\, dx
\\
&\quad 
+\int_\omega\frac{v^{n+1}-v^{n+\frac 1 2}}{\Delta t}\xi\, dy =0
\\
&{\rm with}\ \nabla^{\eta^n}\cdot{\bf u}^{n+1}=0,\quad {\bf u}^{n+1}_{|\Gamma}=v^{n+1}\bn.
\end{aligned}
\end{align*}
Here $\nabla^{\eta}$ is the transformed gradient, ${\bf w}^{n+1/2}$ is the ALE velocity (i.e.\ the time discretization of $\partial_t\bpsi_{\eta^n}$ (see Definition~\ref{def:coor})), and $J^n=\det\nabla\bpsi_{\eta^n}$ is the Jacobian of the transformation from $\Omega_{\eta^n}$ to the reference configuration $\Omega$. Please observe that the above system is a linear equation on a fixed domain and it is solvable as long as $J^n>0$ by the Lax-Milgram Lemma. One can see that no self-intersection implies $J^n>0$. 

Now we define the approximate solutions as a piece-wise constant functions in time:
\begin{equation}
{\bf u}_\dt(t,.)={\bf u}_\dt^n,\ \eta_\dt(t,.)=\eta_\dt^n,\ v_\dt(t,.)=v_\dt^n,\ v^*_\dt(t,.)=v^{n-\frac 1 2}_\dt\text{ for }t\in [n\dt,(n+1)\dt].
\label{aproxNS}
\end{equation}

{\bf Uniform estimates in $\dt$}.

\noindent
The following proposition gives us the uniform boundedness of the approximate solutions defined by \eqref{aproxNS}. It is a consequence of \cite[Lemma 8]{BorSunnon-linear} combined with Lemma~\ref{W14Bound} and Lemma~\ref{H2Bound}.
\begin{proposition}\label{UniformDT}
Let $\dt>0$. Then the approximate solutions defined by \eqref{aproxNS} satisfy the following estimate:
\begin{equation}\label{EstimateDt}
\begin{array}{c}
\|\bu_\dt\|_{L^{\infty}_tL^2_x}+\|\bu_\dt\|_{L^2_tH^1_x}
+\|\eta_\dt\|_{L^{\infty}_tH^2_x}
+\|v_\dt\|_{L^{\infty}_tL^2_x}+\|v^*_\dt\|_{L^{\infty}_tL^2_x}
+\sqrt{\varepsilon}\|\eta_\dt\|_{L^{\infty}_tH^3_x}\leq C,
\end{array}
\end{equation}
where $C$ depends on the data only. Moreover, there exists a $T>0$ independent of $\dt$ such that no self-intersection is approached.
\end{proposition}
\begin{proof}
 The proof can be directly adapted from~\cite[Lemma 8]{BorSunnon-linear} combined with Lemma~\ref{W14Bound} and Lemma~\ref{H2Bound}. In particular, we find by the uniform $L^{\infty}_tH^2_x$ estimates on $\eta_\dt$ that $\norm{\eta_\dt}_{L^\infty_t(L^\infty_x)}$ is uniformly bounded with constants just depending on $\partial \Omega$ and the initial condition.
Moreover, since $v^*_{\dt}$ is bounded in $L^{\infty}_tL^2_x$ we can use the interpolation inequality  for Sobolev spaces to show that there exists $T>0$ such that $\eta_{\dt}$ satisfies \eqref{CoercitivityCondition} and $J_\dt>0$ in $[0,T]$, uniformly in $\dt$ (and $\varepsilon$), cf.~\cite[Proposition 9]{BorSunnon-linear}. In particular $\partial\Omega_{\eta_{\dt}}$ has no self-intersection on $[0,T]$.
\end{proof}
Let us denote by $\bu$, $\eta$, $v$ and $v^*$ the corresponding weak or weak* limits of $\dt\to 0$. From \cite[Lemma 11]{BorSunnon-linear} it follows that $v=v^*$. 

\smallskip
{\bf Compactness for $\dt\to 0$}.

\noindent
First, we prove the strong convergence of the sequence $\eta_\dt$. This is a consequence of the uniform boundedness of the discrete time derivatives $\|\frac{\eta^{n+1}-\eta^n}{\dt}\|_{L^2(\omega)}$ and the boundedness of $\eta_{\dt}$ in $L^{\infty}(0,T;H^3(\omega))$. By using the classical Arzel\'a-Ascoli theorem for the piece-wise affine interpolation we get as in~\cite[Lemma 3]{BorSun} that
$$
\eta_\dt\to\eta\;{\rm in}\;L^{\infty}(0,T;H^s(\omega))\text{ for } s\in(0,3).
$$
This is enough to pass to the limit in the terms connected to the elastic energy. In order to pass to the limit in the convective term and the terms connected to the moving boundary we need strong $L^2$ convergence of $(\bu_\dt,v_\dt)$. This is the most delicate part of the existence proof where one has to use the uniform convergence of $\eta_\dt$ and the fact that the fluid dissipates higher frequencies of the structure velocities.

In the current case this follows by a version of the Aubin-Lions-Simon lemma adapted for the problems with moving boundary \cite[Theorem 3.1. and Section 4.2]{BorSunComp17}. Hence, by passing to the limit we find a $T>0$ such that for every fixed $\varepsilon>0$ there exists a weak solution to~\eqref{WeakSolEqReg}.

\subsection{Proof of Theorem~\ref{thm:ex}.}
In this subsection we first collect the necessary a-priori estimates (which essentially follow from the regularity theorem) and then pass to the limit with $\varepsilon\to 0$. Here the establishment of the non-linearity in the convective term is (as usually) the most delicate part.

{\bf Uniform estimates in $\varepsilon$}.

\noindent
We use the test function: 
\\
{$(\bq,\xi)=\Big (\test\big (D_{-h}^{s}D^{s}_h\eta^{\varepsilon}-\cor(D^s_{-h}D^s_h\eta^{\varepsilon})\big ),D_{-h}^{s}D^{s}_h\eta^{\varepsilon}-\cor(D^s_{-h}D^s_h\eta^{\varepsilon})\Big )$} in \eqref{WeakSolEqReg} in an analogous way as in the proof of Theorem~\ref{RegTem}. In combination with the energy estimates ee obtain the following uniform regularity estimate for all (fixed) $s<\frac{1}2$:
\begin{equation}\label{UniformEps}
\|\eta^{\varepsilon}\|_{L^\infty(0,T;(H^2\cap\sqrt{\varepsilon}H^3)(\omega))}+\|\partial_t\eta^{\varepsilon}\|_{L^\infty(0,T;L^2(\omega))}+\|\eta^{\varepsilon}\|_{L^2(0,T;N^{2+s,2}(\omega))}+\|\partial_t\eta^{\varepsilon}\|_{L^2(0,T;N^{s,2}(\omega))}\leq C.
\end{equation}

{\bf Passing with $\varepsilon \to 0$}.

\noindent
From the (classic) Aubin-Lions lemma we obtain
\begin{equation}\label{StrongConvergenceEta}
\eta^{\varepsilon}\to \eta\;{\rm in}\; L^2(0,T;H^s(\omega))\cap L^{\infty}(0,T;H^{s-1/2}(\omega))
\text{ for } s<5/2.
\end{equation}
In particular, $\eta^{\varepsilon}\to\eta$ in $L^2(0,T;H^2(\omega))\cap L^{\infty}(0,T;L^{\infty}(\omega))$ which is enough to pass to the limit in the elastic terms, see~\eqref{BendingFrom} for the highest order terms. 
The existence result is completed once we can show that (for a sub-sequence)$(\partial_t\eta^{\varepsilon},\bu^{\varepsilon})\to (\partial_t\eta,\bu)$, since this allows to establish all nonlinearities in the limit equation and the existence is complete.

The proof of the $L^2$ convergence of the velocities is known to be the most  delicate part of the construction of weak solutions in the framework of FSI in the incompressible regime, see \cite{CG,LenRuz,BorSunComp17}. Here we present a more universal approach based on the reformulation of the Aubin-Lions lemma (Theorem~\ref{thm:auba}) combined with the extension operator presented in Corollary~\ref{cor:extend}.

\begin{lemma}
\label{lem:l2conv}
There exists a strongly converging subsequence $\epsilon^n\to 0$, such that
\[
(\partial_t\eta^{\varepsilon_n},\bu^{\varepsilon_n})\to (\eta,\bu)
\]
in $L^2_tL^2_x$.
\end{lemma}
\begin{proof}
The strong convergence follows from  Theorem~\ref{thm:auba}. Actually we will apply the theorem two times. First for the boundary compactness and the other time for the interior compactness. {We denote by $\bn_{\eta^{\varepsilon_n}}$ the unit outer normal of $\Omega_{\eta^{\varepsilon_n}}$. From $\diverg\mathrm{Test}_{\eta^{\epsilon_n}}(\partial_t \eta^{\varepsilon_n}-\Kcal_{\eta^{\epsilon_n}}(\partial_t \eta^{\varepsilon_n}))=0$ we conclude that
\[
\Kcal_{\eta^{\epsilon_n}}(\partial_t \eta^{\varepsilon_n})\int_{\partial\Omega_{\eta^{\varepsilon_n}}}\bn(y(x))\cdot \bn_{\eta^{\varepsilon_n}}(x)\, dx=\int_{\partial\Omega_{\eta^{\varepsilon_n}}}\bu^{\varepsilon_n}\cdot \bn_{\eta^{\varepsilon_n}}(x)\, dx=0,
\]
since $\diverg\bu^{\varepsilon_n}=0$ and $\bu^{\varepsilon_n}(x)=\partial_t \eta^{\varepsilon_n}(y(x))\bn(\bfp(x))$, for $x\in\partial\Omega_{\eta^{\varepsilon_n}}$.
Since $\eta^{\varepsilon_n}(\bfp(x))$ is a graph over the surfaces $\Omega$ (see Figure 2) which has a well defined tangent plane almost everywhere, we conclude that $\bn(y(x))\cdot \bn_{\eta^{\varepsilon_n}}(x)>0$ a.e.; this implies that $\Kcal_{\eta^{\epsilon_n}}(\partial_t \eta^{\varepsilon_n})=0$.
And consequently $\testn(\partial_t \eta^{\varepsilon_n})$ is well defined:}
\begin{align*}
\int_0^T\norm{\partial_t \eta^{\varepsilon_n}}_{L^2(\omega)}^2+\norm{
\bu^{\varepsilon_n}}_{L^2(\Omega_{\eta^{\varepsilon_n}}}^2\,dt
&=\int_0^T\skp{\partial_t \eta^{\varepsilon_n},\partial_t \eta^{\varepsilon_n}}+\skp{\bu^{\varepsilon_n},\testn(\partial_t \eta^{\varepsilon_n})}\, dt
\\
&\quad +\int_0^T\skp{\bu^{\varepsilon_n},\bu^{\varepsilon_n}-\testn(\partial_t \eta^{\varepsilon_n})}\, dt=:(I_n)+(II_n).
\end{align*}
For both terms we will show the convergence separately by using Theorem~\ref{thm:auba}
The convergence implies that 
\[
\norm{\partial_t \eta^{\varepsilon_n}}_{L^2_tL^2_x}+\norm{\bu^{\varepsilon_n}\chi_{\Omega_{\eta^{\varepsilon_n}}}}_{L^2_tL^2_x}\to \norm{\partial_t \eta}_{L^2_tL^2_x}+\norm{\bu\chi_{\Omega_{\eta}}}_{L^2_tL^2_x},
\]
which implies, by the uniform convexity of $L^2$, the strong convergence $(\partial_t\eta^{\varepsilon},\bu^{\varepsilon})\to (\eta,\bu)$ and the Lemma is proved.

For the first term $(I_n)$ we define $g_n=(\partial_t \eta^{\varepsilon_n},\bu^{\varepsilon_n}\chi_{\Omega_{\eta^{\varepsilon_n}}})$ and $f_n=(\partial_t \eta^{\varepsilon_n},\testn(\partial_t \eta^{\varepsilon_n}))$, and apply Theorem~\ref{thm:auba} using the spaces $X={L^2(\omega)\times H^{-s}(Q^\kappa)}$ and consequently $X'={L^2(\omega)\times H^{s}(Q^\kappa)}$.  
The space $Z=H^{s_0}(\omega)\times H^{s_0}(Q^\kappa)$ for $0<s<s_0<\frac14$.
Also with respect to time we work in setting of Hilbert spaces, which means that all Lebesgue exponents are equal to two.
Further we recall the smooth extension $E_{\eta^{\varepsilon_n}(t),\delta}$ introduced in Corollary~\ref{cor:extend} and denote by $(\partial_t \eta^{\varepsilon_n})_\delta:=\partial_t \eta^{\varepsilon_n}*\psi_\delta$. This allows to define
\[
f_{n,\delta}:=((\partial_t \eta^{\varepsilon_n})_\delta,E_{\eta^{\varepsilon_n},\delta}(\partial_t \eta^{\varepsilon_n}))
\]

Next, let us check the assumptions of  Theorem~\ref{thm:auba} and Remark~\ref{rem:modification}. First observe that  \cite[Proposition 2.28]{LenRuz} and \eqref{UniformEps}
implies that $g_n$ is uniformly bounded in $L^2_t(H^{s}_x)$ (for $s\leq \frac14$).
Hence the assumptions (1)  follows in a rather straightforward manner by weak compactness in Hilber spaces. 
Next 
(2) follows 
by Corollary~\ref{cor:extend} and the standard estimates for mollifiers:
\[\norm{f_n-f_{n,\delta}}_{L^2(Q^\kappa)}\leq c\norm{\partial_t \eta^{\varepsilon_n}*\psi_\delta-\partial_t \eta^{\varepsilon_n}}_{L^2(\omega)}\leq C\delta^s\norm{\partial_t \eta^{\varepsilon_n}}_{H^s(\omega)}.
\]  
Hence we are left to check (3'). As usual for equi-continuity in time, this is a consequence of the weak formulation of the problem \eqref{WeakSolEqReg}. For $ \sigma\in [t,t+\tau]$ (using the solenoidality and the matching of the extension) we have 
\begin{align*}
\abs{\skp{g_n(t)-g_n(\sigma),f_{n,\delta}(t)}}
&=\absBB{\int_{Q^{\kappa}} \Big(\bu^{\varepsilon_n}(t)\chi_{\Omega_{\eta^{\varepsilon_n}(t)}}-\bu^{\varepsilon_n}(\sigma)\chi_{\Omega_{\eta^{\varepsilon_n}(\sigma)}}\Big)\cdot E_{\eta^{\varepsilon_n}(t),\delta}(\partial_t \eta^{\varepsilon_n}(t))\, dx
\\
&\quad   +\int_{\omega} \big(\partial_t\eta^{\varepsilon_n}(t)-\partial_t\eta^{\varepsilon_n}(\sigma)\big)(\partial_t \eta^{\varepsilon_n}(t))_\delta\big)\, d\by}
\\
&=\absBB{\int_{Q^{\kappa}} \bu^{\varepsilon_n}(t)\chi_{\Omega_{\eta^{\varepsilon_n}(t)}}\cdot E_{\eta^{\varepsilon_n}(t),\delta}(\partial_t \eta^{\varepsilon_n}(t))-\bu^{\varepsilon_n}(\sigma)\chi_{\Omega_{\eta^{\varepsilon_n}(\sigma)}}\cdot E_{\eta^{\varepsilon_n}(\sigma),\delta}(\partial_t \eta^{\varepsilon_n}(t)) dx 
\\
&\quad  +\int_{\omega} \big(\partial_t\eta^{\varepsilon_n}(t)-\partial_t\eta^{\varepsilon_n}(\sigma)\big)(\partial_t \eta^{\varepsilon_n}(t))_\delta\big)\, d\by}
\\
&\quad +\absBB{\int_{Q^{\kappa}} \bu^{\varepsilon_n}(\sigma)\chi_{\Omega_{\eta^{\varepsilon_n}(\sigma)}}\cdot \Big(E_{\eta^{\varepsilon_n}(\sigma),\delta}(\partial_t \eta^{\varepsilon_n}(t)) -E_{\eta^{\varepsilon_n}(t),\delta}(\partial_t \eta^{\varepsilon_n}(t))\Big)dx }
\\
&=: (A_1)+(A_2)
\end{align*}
First observe that by Corollary~\ref{cor:extend}
\begin{align*}
(A_2)&\leq \int \abs{\bu^{\varepsilon_n}(\sigma)}\chi_{\Omega_{\eta^{\varepsilon_n}(\sigma)}}\cdot \int_t^\sigma\abs{\partial_s E_{\eta^{\varepsilon_n}(s),\delta}(\partial_t \eta^{\varepsilon_n}(t))}\, ds \, dx 
\\
&\leq c\norm{\bu^{\varepsilon_n}(\sigma)}_{L^2(\Omega_{\eta^{\varepsilon_n}(\sigma))}}\norm{\partial_t \eta^{\varepsilon_n}(\partial_t \eta^{\varepsilon_n})_\delta}_{L^\infty_t(L^2_x)}\abs{\sigma-t}^\frac12
\\
&\leq C(\delta)\tau^\frac12,
\end{align*}
where we used the uniform $L^\infty_tL^2_x$ estimate of $\bu^{\varepsilon_n}$ and $\partial_t \eta^{\varepsilon_n}$ multiple times.
Second, by the weak formulation \eqref{WeakSolEqReg}, we find
\begin{align*}
(A_1)
&=\absBB{\int_\sigma^t\int_{\Omega_\eta}-\bu^{\varepsilon_n}(s)\cdot \partial_s E_{\eta^{\varepsilon_n}(s),\delta}(\partial_t \eta^{\varepsilon_n}(t))
\\
&+(\sym\nabla\bu^{\varepsilon_n}(s)-\bu^{\varepsilon_n}(s) \otimes \bu^{\varepsilon_n} (s)) :\nabla  E_{\eta^{\varepsilon_n}(s),\delta}(\partial_t \eta^{\varepsilon_n}(t))\, dx 
\\
& 
\quad +\int_{\omega}a_m(t,\eta^{\varepsilon_n},(\partial_t \eta^{\varepsilon_n}(t))_\delta)+a_b(t,\eta^{\varepsilon_n},(\partial_t \eta^{\varepsilon_n})_\delta)
+{\varepsilon_n}\nabla^3\eta^{\varepsilon_n}(s):\nabla^3(\partial_t \eta^{\varepsilon_n}(t))_\delta\, d\by\, ds}
\\
&\leq C(\delta)\Big((\norm{\bu^{\varepsilon_n}}_{L^2(t,t+\tau;W^{1,2}+1)(\Omega_{\eta^{\varepsilon_n}})}(\norm{\partial_t \eta^{\varepsilon_n}}_{L^\infty_t(L^2_x)}+\norm{\eta^{\varepsilon_n}}_{L^\infty_t(H^2_x+\varepsilon H^3_x)})
\\
&\quad +\norm{\eta^{\varepsilon_n}}_{L^2(t,t+\tau;W^{1,\infty}(\omega)}\norm{\bu^{\varepsilon_n}}_{L^\infty(0,T;L^{2}(\Omega_{\eta^{\varepsilon_n}}))}^2\Big) \abs{t-\sigma}^\frac12 
\\
&\leq C(\delta)\tau^\frac12.
\end{align*}
This implies (3'), namely
\[
\absB{\dashint_t^{t+\tau}\skp{g_n(t)-g_n(\sigma),f_{n,\delta}(t)}\, d\sigma}\leq C(\delta)\tau^\frac12.
\]

This finishes the proof of the convergence of $(I)_n$ term. For the second term $(II_n)$ we again apply Theorem~\ref{thm:auba}.
Here we set $g_n=\bu^{\varepsilon_n}\chi_{\Omega_{\eta^{\varepsilon_n}}}$ and $f_n=(\bu^{\varepsilon_n}-\testn(\partial_t \eta^{\varepsilon_n}))\chi_{\Omega_{\eta^{\varepsilon_n}}}$ 
We apply Theorem~\ref{thm:auba} with the following spaces $X={H^{-s}(Q^\kappa)}$ and consequently $X'={H^{s}(Q^\kappa)}$ for some $s\in (0,\frac14)$. Further we define $Z=L^2(Q^\kappa)$. Please observe that we may extend all involved quantities by zero to be functions over $Q^\kappa$. Finally, we again set all Lebesgue exponents to two. Similarly as for the first term, again the main effort is the construction of the right mollification. Indeed the assumptions $(1)$ and $(4)$ follow by standard compactness arguments. In particular, for assumption $(1)$ it has been shown in \cite[Proposition 2.28]{LenRuz} that $g_n$ is uniformly bounded in $H^{s}(Q^\kappa)$ (if $s\leq \frac14$).
For (2) we use the fact that $f_n$ has a zero trace on $\partial\Omega_{\eta^{\varepsilon_n}(t)}$. First, let $\delta>0$ be given. We take $n_\delta$ large enough and $\tau_\delta>0$ small enough, such that 
\begin{align}
\label{eq:domain}
\sup_{n\geq n_\delta}\sup_{\tau\in (t-\tau_\delta,t+\tau_\delta)\cap [0,T]}\norm{\eta(t,x)-\eta^{\varepsilon_n}(\tau,x)}_\infty\leq \delta.
\end{align}
Second, we fix $0<s_0<s$ and $\epsilon>0$. By \cite[Lemma A.13]{LenRuz} there exits a $\sigma_\epsilon$ and a sequence $\tilde{f}_{n,\delta}$, such that $\supp(\tilde{f}_{n,\delta}(t))\subset \Omega_{\eta^{\varepsilon_n}(t)-3\delta}$ for all $3\delta\leq \sigma_\epsilon$,
 that is divergence free and 
$\norm{f-\tilde{f}_{n,\delta}}_{(H^{-{s_0}}(Q^\kappa))}\leq \epsilon\norm{f}_{L^2(Q^\kappa)}$. We mollify this solenoidal function to define
\[
f_{n,\delta}= \tilde{f}_{n,\delta}*\psi_\delta\text{ where $\psi$ is the standard mollifier in space.}
\]
Then this definition implies (3') by standard mollification estimate. Indeed, choosing 
\begin{align*}
\norm{f_n-f_{n,\delta}}_{H^{-s}(Q^\kappa)}&\leq \norm{\tilde{f}_{n,\delta}-f_{n,\delta}}_{H^{-s}(Q^\kappa)} + \norm{\tilde{f}_{n,\delta}-f_n}_{H^{-s_0}(Q^\kappa)}
\\
&\leq c\delta^{s-s_0}\norm{\tilde{f}_{n,\delta}}_{H^{-s_0}(Q^\kappa)}+\epsilon\norm{f_n}_{L^2(Q^\kappa)}\leq c\epsilon \norm{f_n}_{L^2(Q^\kappa)},
\end{align*}
for $\delta$ small enough in dependence of $s-s_0$.
Please observe that by the properties of the mollification and \eqref{eq:domain} that $\supp(f_{n,\delta})\subset \Omega_{\eta^{\varepsilon_n}(t)-2\delta}\subset \Omega_{\eta^{\varepsilon_m}}(\tau)$ which implies that $f_{n,\delta}(t)$ can be used as a testfunction on the fluid equation alone over the interval $[t,t+\tau]$ for $t\in [0,T-\tau]$. Hence (3') follows exactly along the lines of the above estimates using the weak formulation \eqref{WeakSolEqReg}. Notice that here we will work just with the fluid equation since the traces of test function are zero at the moving interface. 
\end{proof}

{\bf End of the proof of Theorem~\ref{thm:ex} }

\noindent
The a-priori estimates and the above compactness arguments guarantee that for given initial conditions there is a minimal time interval $T>0$ for which a weak solution exists (see Remark~\ref{rem:T}). Once the solution is established we can repeat the argument (by using $\eta(T),\partial_t\eta(T),\bu(T)$ as initial conditions) until either a self-intersection is approached or a degeneracy of the $H^2$-coercivity is violated (namely if  $\gamma(\eta(t,x))\to 0$ for some $t\to T$).

\section*{Acknowledgments}
The work of B. Muha was supported by the Croatian Science Foundation (Hrvatska zaklada za znanost) grant number IP-2018-01-3706. S. Schwarzacher thanks the support of the research support programs of Charles University: PRIMUS/19/SCI/01 and UNCE/SCI/023. S.\ Schwarzacher thanks the support of the program GJ19-11707Y of the Czech national grant agency (GA\v{C}R).


\end{document}